\documentclass[12pt]{amsart}
\usepackage[margin=1.2in]{geometry}
\usepackage{graphicx,latexsym}
\usepackage{amsfonts, amssymb, amsmath, amsthm}
\usepackage{amsbsy}

\newcommand{\N}{\mathbb{N}}
\newcommand{\Z}{\mathbb{Z}}
\newcommand{\R}{\mathbb{R}}
\newcommand{\C}{\mathbb{C}}

\newcommand{\Eh}{\mathcal{E}^{\{M_p\},h}}
\newcommand{\Ek}{\mathcal{E}^{\{M_p\},k}}

\newcommand{\El}{\mathcal{E}^{\{M_p\},l}}
\newcommand{\Dh}{\mathcal{D}^{\{M_p\},h}}
\newcommand{\Df}{\mathcal{D}^{(M_p)}}
\newcommand{\Ef}{\mathcal{E}^{(M_p)}}
\newcommand{\De}{\mathcal{D}^{\{M_p\}}}
\newcommand{\Ee}{\mathcal{E}^{\{M_p\}}}
\newcommand{\Ds}{\mathcal{D}^{\ast}}
\newcommand{\Es}{\mathcal{E}^{\ast}}
\newcommand{\Alg}{\mathcal{A}^{\ast,\dagger}}

\newcommand{\Nf}{{\mathcal{E}^{(M_p)}_{\mathcal{N}}}}
\newcommand{\Ne}{{\mathcal{E}^{\{M_p\}}_{\mathcal{N}}}}
\newcommand{\Ns}{{\mathcal{E}^{\ast}_{\mathcal{N}}}}
\newcommand{\Efc}{{\mathcal{E}^{(M_p)}_{\mathcal{M}}}}
\newcommand{\Eec}{{\mathcal{E}^{\{M_p\}}_{\mathcal{M}}}}
\newcommand{\Esc}{{\mathcal{E}^{\ast}_{\mathcal{M}}}}
\newcommand{\Gss}{\mathcal{G}^{\ast}}
\newcommand{\Gsi}{\mathcal{G}^{\ast,\infty}}

\newcommand{\dxi}{{\rm d}\xi }
\newcommand{\dt}{{\rm d}t }
\newcommand{\deta}{{\rm d}\eta }

\newcommand{\Ft}{\mathcal{F}}

\newtheorem{theorem}{Theorem}[section]
\newtheorem{proposition}[theorem]{Proposition}
\newtheorem{lemma}[theorem]{Lemma}
\newtheorem{corollary}[theorem]{Corollary}

\theoremstyle{definition}

\theoremstyle{remark}
\newtheorem{remark}[theorem]{Remark}

\numberwithin{equation}{section}

\begin{document}
\title[Embeddings of ultradistributions into algebras]{Optimal embeddings of ultradistributions into differential algebras}

\author[A. Debrouwere]{Andreas Debrouwere}
\thanks{A.\ Debrouwere gratefully acknowledges support by Ghent University, through a BOF Ph.D.-grant.}
\address{Department of Mathematics, Ghent University, Krijgslaan 281, 9000 Gent, Belgium}
\email{Andreas.Debrouwere@UGent.be}

\author[H. Vernaeve]{Hans Vernaeve}
\thanks{H.\ Vernaeve is supported by grant 1.5.138.13N of the Research Foundation Flanders FWO}
\address{Department of Mathematics, Ghent University, Krijgslaan 281, 9000 Gent, Belgium}
\email{hvernaev@cage.UGent.be}

\author[J. Vindas]{Jasson Vindas}
\thanks{J. Vindas gratefully acknowledges support by Ghent University, through the BOF-grant 01N01014.}
\address{Department of Mathematics, Ghent University, Krijgslaan 281, 9000 Gent, Belgium}
\email{Jasson.Vindas@UGent.be}


\begin{abstract}
We construct embeddings of spaces of non-quasianalytic ultradistributions into differential algebras enjoying optimal properties in view of a Schwartz type impossibility result, also shown in this article. We develop microlocal analysis in theses algebras consistent with the  microlocal analysis in the corresponding spaces of ultradistributions.
\end{abstract}

\maketitle

%
%
%

\section{Introduction}
Differential algebras of generalized functions containing the space of Schwartz distributions have been constructed starting with the work of J.F. Colombeau \cite{Colombeau84,Colombeau85} (see also the recent work \cite{a-j-c}). They provide a framework in which nonlinear equations and equations with strongly singular data or coefficients can be solved and in which their regularity can be analyzed. The theory has found diverse applications in the study of partial differential equations \cite{Hormann-DeHoop,Hormann-Ober-Pilipovic,Nedel,Ober}, variational problems \cite{garetto-vernaeve}, differential geometry \cite{Kunzinger-Steinbauer}, and relativity theory (see \cite{GGK} and the references therein). 

The aim of this paper is to further develop the nonlinear theory of generalized functions in the context of ultradistributions. The need for a nonlinear theory of ultradistributions naturally occurs e.g.\ if one considers hyperbolic systems with rapidly growing nonlinear terms \cite{Gramchev}, or in recent studies on well-posedness of weakly hyperbolic equations with time dependent nonregular coefficients \cite{garetto2015,garetto-r2015}. We construct in this article new embeddings of ultradistributions into differential algebras enjoying optimal properties in the sense we now proceed to explain. Attempts to solve this problem have been made by many authors, but none of the so far proposed embeddings enjoys the optimal properties we will establish here.  

According to Schwartz' impossibility result \cite{GGK,Schwartz}  there does not exist an associative, commutative differential algebra, containing the space of distributions, in which the multiplication of $C^{k}$-functions ($k<\infty$) coincides with the ordinary multiplication. However, in the algebra constructed by J.F. Colombeau, the multiplication of  $C^{\infty}$-functions, the natural class used to define the Schwartz distributions, is preserved. The latter fact is crucial in order to develop regularity and microlocal analysis in this nonlinear context \cite{Dapic,Hormann-Ober-Pilipovic,Ober}. Colombeau's embedding of distributions is then optimal in the sense of the Schwartz impossibility result.
 Over the last years many authors have defined differential algebras containing spaces of (non-quasianalytic) ultradistributions \cite{Benmeriem2,Delcroix2004,Delcroix,Gramchev,PilScar}. However, in any of the algebras constructed so far, containing the space of  $M_p$-ultradistributions \cite{Komatsu}, there has always been the rather unnatural restriction that the multiplication is preserved only for some ultradifferentiable functions of a certain strictly more regular Denjoy-Carleman class than $M_p$. Developing a non\-linear theory for non-quasianalytic ultradistributions that avoids this loss of regularity phenomenon has up to now been an open problem.

In this article we resolve this problem. More concretely, we show that it is possible to embed the spaces of $M_p$-ultradistributions into differentiable algebras in such a way that the multiplication of all $M_p$-ultradifferentiable functions is preserved and, moreover, by establishing an analogue of Schwartz' impossibility result for ultradistributions, we show that our embedding is optimal. It turns out that our ideas also apply to develop a similar theory for $\omega$-ultradistributions, that is, ultradistributions defined via weight functions \cite{Bjorck,b-m-t,Cior}. We treat both the Beurling and Roumieu cases of weight sequences and weight functions.

Furthermore, we introduce a notion of regularity in our algebras of generalized functions which coincides with ultradifferentiability when restricted to ultradistributions, and we develop microlocal analysis in our algebras, similar to the one developed in Colombeau's algebra \cite{Dapic, Nedel}, which has served as a framework for the study of the propagation of singularities under the action of differential operators with singular coefficients (see e.g. \cite{Hormann-Ober-Pilipovic}).

This paper is organized as follows. Sections \ref{preli}--\ref{section wavefronts}  deal with the case of generalized functions defined via weight sequences; in Section \ref{section weight functions} we explain the necessary modifications to achieve similar results for weight functions.  In the preliminary Section \ref{preli} we fix the notation and explain some of the spaces of ultradifferentiable functions and ultradistributions to be considered in this work. The analogue of Schwartz' impossibility result for ultradistributions is stated and proved in Section \ref{imposs section}. The construction of our differential algebras of generalized functions as well as some of their basic sheaf-theoretic properties are discussed in Section \ref{section algebra}. We provide there a null characterization of the spaces of negligible nets, which we apply to obtain pointwise characterizations of our generalized functions. We mention that in the Roumieu case the definition of our algebras differs considerably from those previously proposed in the literature (cf. Remark \ref{def Roumieu}); the projective description of our algebras of Roumieu type obtained in Subsection \ref{subsection projective} may serve to explain that our new definition is natural (see \cite{Komatsu3, Pil94} for analogues in the theory of ultradistributions). 
The embedding of ultradistributions into our algebras is accomplished in Section \ref{section-embedding}, where we also show that the multiplication of $M_p$-ultradifferentiable functions is preserved under this embedding. The next two sections are devoted to local and microlocal analysis in our newly defined algebras. In Section \ref{section regularity}, we study regularity in our setting and show that an $M_p$-ultradistribution is a regular generalized function if and only if it is an $M_p$-ultradifferentiable function. Based upon this notion of regularity we introduce in Section \ref{section wavefronts} generalized wave front sets for our generalized functions and show that this microlocal notion is consistent with the one for ultradistributions  \cite{Komatsumla, Pilipovic}.
\section{Preliminaries}\label{preli}
In this section we fix the notation and briefly explain some properties of spaces of ultradifferentiable functions and ultradistributions defined via weight sequences \cite{Komatsu,Roumieu,Carmichael}. These spaces will be employed throughout Sections  \ref{imposs section}--\ref{section wavefronts} of the article; we will discuss the Beurling-Bj\"{o}rck approach \cite{Bjorck} to ultradistribution theory via weight functions in Section \ref{section weight functions}.  
We fix a positive weight sequence $(M_p)_{p \in \N}$ with $M_0 = 1$. We shall always assume that $M_p$ satisfies:
\begin{enumerate}
\item[$(M.1)\:$] $M_p^2 \leq M_{p-1}M_{p+1}$, $p \in \Z_+$,
\item[$(M.2)\:$] $M_{p+q} \leq AH^{p+q}M_pM_q$, $p,q \in \N$, for some $A,H \geq 1$,
\item[$(M.3)'$] $\displaystyle \sum_{p = 1}^\infty \frac{M_{p-1}}{M_p} < \infty$.
\end{enumerate}
The \emph{associated function} of $M_p$ is defined as
\[M(t) = \sup_{p \in \N} \log \left(\frac{t^p}{M_p}\right), \qquad t \geq 0.\]
We refer to \cite{Komatsu} for the meaning of these three conditions and their translation in terms of the associated function. In particular, under $(M.1)$, the assumption $(M.2)$ holds \cite[Prop 3.6]{Komatsu} if and only if
\begin{equation}
2M(t) \leq M(Ht) + \log A, \qquad \forall t \geq 0,
\label{assM.2}
\end{equation}
where $A,H$ are the constants witnessing $(M.2)$, while $(M.3)'$ becomes equivalent \cite[Lemm. 4.1]{Komatsu} to
\begin{equation}
\int_1^\infty \frac{M(t)}{t^2}\dt < \infty.
\label{assM.3}
\end{equation}
As usual, the relation $M_p\subset N_p$  between two weight sequences means that there are $C,l>0$ such that 
$M_p\leq Cl^{p}N_{p},$ $p\in\mathbb{N}$, while the stronger relation $M_p\prec N_p$ means that the latter inequality remains valid for every $l>0$ and suitable $C=C_{l}>0$. 

Let $\Omega \subseteq \R^d$ be open. For $K \Subset \Omega$ (a compact subset of $\Omega$·) and $h > 0$ the space $\Eh(K)$ consists of all $\phi \in C^\infty(\Omega)$ such that
\[ \| \phi \|_{\Eh(K)} := \sup_{\substack{x \in K \\ \alpha \in \N^d}} \frac{|\phi^{(\alpha)}(x)|}{h^{|\alpha|}M_{|\alpha|}} < \infty. \]
We define the spaces of $M_p$-ultradifferentiable functions as
\[ \Ef(\Omega) = \varprojlim_{K \Subset \Omega} \varprojlim_{h \rightarrow 0} \Eh(K), \quad \Ee(\Omega) = \varprojlim_{K \Subset \Omega} \varinjlim_{h \rightarrow \infty} \Eh(K).  \]
The subspace of $\Eh(K)$ consisting of elements with supports in $K$ is denoted as $\Dh(K)$ and we set 
\[ \Df(K) =\varprojlim_{h \rightarrow 0} \Dh(K), \qquad \De(K) = \varinjlim_{h \rightarrow \infty} \Dh(K),  \]
and
\[ \Df(\Omega) = \varinjlim_{K \Subset \Omega}  \Df(K), \qquad \De(\Omega) = \varinjlim_{K \Subset \Omega} \De(K).  \]
Naturally, the non-triviality of these spaces of compactly supported functions is equivalent to $(M.3)'$. 
Their duals ${\Df}'(\Omega)$ and ${\De}'(\Omega)$ are the ultradistribution spaces of class $(M_p)$  (Beurling type) and class $\{M_p\}$ (Roumieu type), respectively. The dual spaces ${\Ef}'(\Omega)$ and ${\Ee}'(\Omega)$ correspond to the compactly supported ultradistributions. As customary, one employs the notation $\ast= (M_p)$ or $\{M_p\}$ to treat both cases simultaneously. In addition, we shall often first state assertions for the Beurling case followed in parenthesis by the corresponding statements for the Roumieu case.

An $(M_p)$-ultradifferential operator ($\{M_p\}$-ultradifferential operator) is an infinite order differential operator
\[ P(D) = \sum_{\alpha \in \N^d} a_\alpha D^\alpha, \qquad a_\alpha \in \C,\] 
($D^{\alpha}=(-i \partial)^{\alpha}$) where the coefficients satisfy the estimate 
\[|a_\alpha| \leq \frac{CL^{|\alpha|}}{M_{|\alpha|} }\]
for some $L > 0$ and $C > 0$ (for every $L>0$ there is $C=C_{L} > 0$).
Condition $(M.2)$ ensures that $P(D)$ acts continuously on $\Ds(\Omega)$ and $\Es(\Omega)$ and hence it can be defined on the corresponding ultradistribution spaces. The symbol $P(\xi) = \sum_{\alpha \in \N^d} a_\alpha \xi^\alpha$ is of course an entire function, it is called an $\ast$-ultrapolynomial.
Let $q$ be a polynomial; one has the following version of Leibniz' rule,
\begin{equation}\label{equltrapolpol} P(D)( q f ) = \sum_{\beta \leq \deg q} \frac{1}{\beta!}D^{\beta}q \cdot [(D^\beta P)(D)f], \qquad \forall f \in {\Ds}'(\Omega). 
\end{equation}
We fix constants in the Fourier transform as
\[ \Ft(f)(\xi) = \widehat{f}(\xi) = \int_{\R^d} f(x)e^{ix\cdot \xi} dx. \]
\section{ Impossibility result on the multiplication of ultradistributions}\label{imposs section}
In this section we show an analogue of Schwartz' famous impossibility result \cite{Schwartz} in the setting of ultradistributions. Recall that in the case of Schwartz distributions this result asserts that it is impossible to embed $\mathcal{D}'(\Omega)$ into an associative and commutative differential algebra preserving differentiation, the unity function (= constant function 1), and at the same time the (pointwise) multiplication of continuous functions -- or more generally $C^{k}$-functions for any $k\in\mathbb{N}$, see \cite[p. 7]{GGK}. The role of the continuous functions in our impossibility result is played by a space $\mathcal{E}^{\dagger}(\Omega)$ with slightly less regularity than $\mathcal{E}^{\ast}(\Omega)$ (recall $\ast$ always denotes the Beurling or Roumieu case of $M_p$). We assume in the rest of this section that $N_p$ is a weight sequence that satisfies $(M.1)$ and stability under differential operators \cite{Komatsu}, namely, the condition
\begin{itemize}
 \item[$(M.2)'$] $N_{p+1} \leq AH^pN_p$, $p \in \N$, for some $A,H \geq 1$.
\end{itemize}
We set $\dagger=(N_p)$ or $\{N_p\}$. When embedding ${\Ds}'(\Omega)$ into some associative and commutative algebra $(\mathcal{A}^{\ast,\dagger}(\Omega), +, \circ)$, the following requirements appear to be natural:
\begin{enumerate}
\item[$(P.1)$] ${\Ds}'(\Omega)$ is linearly embedded into $\Alg(\Omega)$ and $f(x) \equiv 1$ is the unity in $\Alg(\Omega)$.
\item[$(P.2)$] For each $\ast$-ultradifferential operator $P(D)$ there is a linear operator $P(D):  \Alg(\Omega) \rightarrow \Alg(\Omega)$ satisfying (cf. (\ref{equltrapolpol}))
 \[ 
 P(D)( q \circ f ) = \sum_{\beta \leq \deg q} \frac{1}{\beta!} D^\beta q \circ (D^\beta P)(D)f, \qquad \forall f \in \Alg(\Omega), 
 \]
and every polynomial $q$. Moreover, $P(D)_{|{\Ds}'(\Omega)}$ coincides with the usual action of $P(D)$ on $\ast$-ultradistributions.
\item[$(P.3)$] $\circ_{|\mathcal{E}^\dagger(\Omega) \times \mathcal{E}^\dagger(\Omega)}$ coincides with the pointwise product of functions.
\end{enumerate}
\smallskip

The ensuing result imposes a limitation on the possibility of constructing such an algebra.

\begin{theorem} \label{impossibility} Let $M_p \prec N_p$. Then, there is no associative and commutative algebra $\mathcal{A}^{\ast,\dagger}(\Omega)$ satisfying $(P.1)$--$(P.3)$. Moreover, in the Beurling case of $\ast$, there does not exist an algebra  $\mathcal{A}^{(M_p),\{M_p\}}(\Omega)$ either satisfying $(P.1)$--$(P.3)$.
\end{theorem}
\begin{proof} Suppose $\mathcal{A}^{\ast,\dagger}(\Omega)$ is such an algebra (we treat all cases simultaneously). One can generalize Schwartz' original idea by making use of the following observation: Given a $\ast$-ultradifferential operator $P(D)$, $g \in  \mathcal{E}^\dagger(\Omega)$, and a polynomial $q$, one has that $q \circ P(D)g =  q P(D)g$ -- which can be readily shown by induction on the degree of $q$. Assume for simplicity that $0 \in \Omega$. Write $H(x)= H(x_1, \ldots, x_d): =  H(x_1) \otimes \cdots \otimes H(x_d)$, where $H(x_j)$ denotes the Heaviside function, that is, the characteristic function of the positive half-axis, and  $\operatorname*{p.v.}(x^{-1}) :=  \operatorname*{p.v.}(x_1^{-1}) \otimes \cdots \otimes  \operatorname*{p.v.}(x_d^{-1})$, where $\operatorname*{p.v.}(x_j^{-1})$ stands for the principle value regularization of the function $x_j^{-1}$.  Let $f_j$ denote either $H(x_j)$ or $\operatorname{p.v. }(x_j^{-1})$ and  let $\Omega_j$ be the projection of $\Omega$ onto the $x_j$-axis. Note that we can obviously decompose $f_{j}$ as $f_j = f_{j,1}+ g_{j,2}$, where $f^1_{j} \in {\Es}'(\Omega_j)$ and $g_{j,2}  \in \Es(\Omega_j)$. By employing the structural theorem for ultradistributions \cite{Taki}, we can find $g_{j,1} \in \mathcal{E}^\dagger(\Omega_j)$ and a $\ast$-ultradifferential operator $P_{j,1}(D_j)$ such that $P_{j,1}(D_j)g_{j,1} = f_{j,1}$, so that $f_{j}=P_{j,1}(D_j)g_{j,1}+ P_{j,2}(D_j)g_{j,2}$, where $P_{j,2}(D_j)$ is the trivial ultradifferential operator $P_{j,2}(D_j)= 1$. We conclude that $H$ and $\operatorname*{p.v.}(x^{-1})$  are linear combinations of ultradistributions of the form
 \[P_{1,k_1}(D_1)g_{1,k_1} \otimes \cdots  \otimes P_{d,k_d}(D_d)g_d^{k_d} = P_{1,k_1}(D_1) \cdots P_{d,k_d}(D_d) (g_1^{k_1} \otimes \cdots  \otimes g_d^{k_d}),\]
where $k_j \in \{1,2\}$.   Condition $(M.2)$ ensures the class of $\ast$-ultradifferential operators is closed under composition and therefore $H$ and $\operatorname{p.v. }(x^{-1})$ are linear combinations of terms of the form $P(D)g$, where $P(D)$ is a $\ast$-ultradifferential operator and $g \in \mathcal{E}^\dagger(\Omega)$.  Set
\[
\boldsymbol{\partial} =  \frac{\partial^d}{\partial x_1 \cdots \partial x_d} \quad \mbox{and} \quad q(x)=x_1x_2\cdots x_d\:;
\]
the observation made at beginning of the proof now yields $q\circ \boldsymbol{\partial} H = (q \boldsymbol\partial H)$ and $q \circ \operatorname{p.v. }(x^{-1}) = q \operatorname{p.v. }(x^{-1})$. Since $q \boldsymbol{\partial} H = 0$ and $ q \operatorname{p.v. }(x^{-1}) = 1$ in  ${\Ds}'(\Omega)$, we obtain
\[ 
 \boldsymbol{\partial} H= \boldsymbol{\partial} H \circ ( q \circ \operatorname*{p.v.}(x^{-1})) =  (  \boldsymbol{\partial} H \circ  q) \circ \operatorname*{p.v.}(x^{-1}) = 0,
\]
contradicting $\boldsymbol{\partial} H = \delta \neq 0$ in ${\Ds}'(\Omega)$ and the injectivity of ${\Ds}'(\Omega)\to \Alg(\Omega)$.
\end{proof}
Despite this impossibility theorem, we shall construct an algebra $\Gss(\Omega)$ that does satisfy the desirable properties $(P.1)$--$(P.3)$ for $\dagger = \ast$, and our embedding of ${\Ds}'(\Omega)$ will therefore be optimal in this sense.

\section{Algebras of generalized functions of class $(M_p)$ and $\{M_p\}$}\label{section algebra}
We now introduce algebras  $\mathcal{G}^{\ast}(\Omega)$ of generalized functions of class $\ast$, where $\ast=(M_p)$ or $\{M_p\}$, as quotients of algebras consisting of nets of ultradifferentiable functions, indexed by a parameter $\varepsilon\in(0,1]$ and satisfying an appropriate growth condition as $\varepsilon \to 0^+$. Their construction is of course
inspired by that of the special Colombeau algebra $\mathcal{G}(\Omega)$, and a comparison between definitions might be useful for  the reader (see e.g. \cite[Chap. 1]{GGK}). This section is dedicated to the study of various useful properties of our factor algebras, including a ``null characterization" of the so-called negligible elements, basic sheaf-theoretic properties, an alternative projective type description of  $\mathcal{G}^{\ast}(\Omega)$ in the Roumieu case, associated rings of generalized numbers, and point value characterizations of our generalized functions. The embedding of ultradistributions is discussed in Section \ref{section-embedding}.
\subsection{Definition and basic properties.}We begin by introducing the following spaces of \emph{$\ast$-moderate} nets of ultradifferentiable functions,
\begin{align*}
\Efc(\Omega) = \{ (f_\varepsilon)_\varepsilon \in \Ef(\Omega)^{(0,1]} \, : \, &(\forall K \Subset \Omega) (\forall h > 0) (\exists k > 0) \\
 &\|f_\varepsilon\|_{\Eh(K)}  = O(e^{M(k/\varepsilon)})         \mbox{ as }  \varepsilon \to 0^+  \},
\end{align*}
\begin{align*}
\Eec(\Omega) = \{ (f_\varepsilon)_\varepsilon \in \Ee(\Omega)^{(0,1]} \, : \, &(\forall K \Subset \Omega) (\forall k > 0) (\exists h > 0) \\
 &\|f_\varepsilon\|_{\Eh(K)}  = O(e^{M(k/\varepsilon)})         \mbox{ as }  \varepsilon \to 0^+  \},
\end{align*}
and the spaces of $\ast$-\emph{negligible} elements
\begin{align*}
\Nf(\Omega) = \{ (f_\varepsilon)_\varepsilon \in \Ef(\Omega)^{(0,1]} \, : \, &(\forall K \Subset \Omega) (\forall h >0) (\forall k > 0) \\
 &\|f_\varepsilon\|_{\Eh(K)}  = O(e^{-M(k/\varepsilon)})         \mbox{ as }  \varepsilon \to 0^+  \},
\end{align*}
\begin{align*}
\Ne(\Omega) = \{ (f_\varepsilon)_\varepsilon \in \Ee(\Omega)^{(0,1]} \, : \, &(\forall K \Subset \Omega) (\exists k > 0) (\exists h > 0) \\
 &\|f_\varepsilon\|_{\Eh(K)}  = O(e^{-M(k/\varepsilon)})         \mbox{ as }  \varepsilon \to 0^+  \}.
\end{align*}

\begin{remark}\label{def Roumieu}
These definitions should be compared with those occurring in other works dealing with constructions of generalized function algebras based on nets of ultradifferentiable functions (cf. \cite{Benmeriem2,Benmeriem3,Delcroix2004,Delcroix,Gramchev,PilScar}). The most significant difference lies in our completely different choice of quantifiers in the Roumieu case. It is important to point out that this choice of quantifiers will play an essential role in the next section when embedding ultradistributions and preserving the product of $\ast$-ultradifferentiable functions.
\end{remark}

\smallskip
Note that the conditions $(M.1)$ and $(M.2)$ (cf. (\ref{assM.2})) obviously ensure that $\Esc(\Omega)$ is an algebra under pointwise multiplication of nets, and that $\Ns(\Omega)$ is an ideal of $\Esc(\Omega)$. Hence, we can define the algebra $\Gss(\Omega)$ of \emph{generalized functions of class $\ast$} (in short: $\ast$-generalized functions) as the factor algebra
\[ \Gss(\Omega) = \Esc(\Omega)/\Ns(\Omega). \]

We denote by $[(f_\varepsilon)_\varepsilon]$ the equivalence class of $(f_\varepsilon)_\varepsilon\in\Esc(\Omega)$.
Observe that $\Es(\Omega)$ can be regarded as a subalgebra of $\Gss(\Omega)$ via the constant embedding:
\begin{equation}
\label{eqconstantembedding}
\sigma(f): = [(f)_\varepsilon], \qquad  f \in \Es(\Omega).
\end{equation}
We also remark that $\Gss(\Omega)$ can be endowed with a canonical action of $\ast$-ultradifferential operators. In fact, since $\ast$-ultradifferential operators act continuously on $\Es(\Omega)$, we have that  $\Esc(\Omega)$ and $\Ns(\Omega)$ are closed under $\ast$-ultradifferential operators if we define their actions on nets as $P(D)((f_{\varepsilon})_{\varepsilon}):=(P(D)f_{\varepsilon})_{\varepsilon}$. Consequently, every $\ast$-ultradifferential operator $P(D)$ canonically induces a linear operator 
$$
P(D):\Gss(\Omega)\to \Gss(\Omega),
$$
which clearly satisfies the generalized Leibniz rule (\ref{equltrapolpol}) for any $f\in\Gss(\Omega)$ and the polynomial $q$ seen as $\sigma(q)$. 

We now show the null characterization of the ideal $\Ns(\Omega)$. This result is a very useful tool and will be constantly applied through the rest of the article.

 \begin{proposition} \label{nullchar}
 Let $(f_\varepsilon)_\varepsilon \in \Esc(\Omega)$. Then, $(f_\varepsilon)_\varepsilon \in \Ns(\Omega)$ if and only if 
\[(\forall K \Subset \Omega) ( \forall k > 0)\qquad ((\forall K \Subset \Omega) ( \exists k > 0))    \]
\[ \sup_{x \in K}|f_\varepsilon(x)| = O(e^{-M(k/\varepsilon)}). \]
 \end{proposition}
 \begin{proof}
 For the proof of this proposition, we establish the following multivariate version of the Landau-Kolmogorov inequality,
\begin{equation}
\label{Landau-Kolmogorov inequality} \sup_{|\alpha| = k} \| f^{(\alpha)}\|_{L^\infty(\R^d)} \leq 2\pi d^k \| f \|^{1 - k/n}_{L^\infty(\R^d)}\sup_{|\alpha| = n} \| f^{(\alpha)}\|^{k/n}_{L^\infty(\R^d)}, \quad 0 < k < n, 
\end{equation}
which we claim to be valid for any $f \in C^\infty(\R^d)$ such that $f^{(\alpha)} \in L^\infty(\R^d)$ for every $|\alpha|\leq n$. We prove (\ref{Landau-Kolmogorov inequality}) below, but let us momentarily assume it and show how the conclusion of the proposition follows from it.
 
 Suppose $(f_\varepsilon)_\varepsilon$ satisfies the $0$-th order estimate. Let $K \Subset \Omega$  and choose $\psi \in \Ds(\Omega)$ such that $\psi \equiv 1$ on a neighborhood of $K$. Set  $(g_\varepsilon)_\varepsilon = (\psi f_\varepsilon)_\varepsilon \in \Esc(\Omega)$ and $K' = \operatorname{supp} \psi$. We have
\[(\forall h_1 > 0)(\exists k_1 > 0)(\exists C_1 > 0)(\exists \varepsilon_1 > 0) \quad ((\forall k_1 > 0)(\exists h_1 > 0)(\exists C_1 > 0)(\exists \varepsilon_1 > 0))    \]
\[ \| g_\varepsilon^{(\alpha)}\|_{L^\infty(\R^d)} \leq C_1 h_1^{|\alpha|}M_{|\alpha|}e^{M(k_1/\varepsilon)}, \qquad \forall \alpha \in \N^d, \varepsilon \in (0,\varepsilon_1),  \] 
and
\[(\forall k_2 > 0)(\exists C_2 > 0)(\exists \varepsilon_2 > 0) \qquad  ((\exists k_2 > 0)(\exists C_2 > 0)(\exists \varepsilon_2 > 0))    \]
\[ \sup_{x \in K'}|f_\varepsilon(x)| \leq C_2e^{-M(k_2/\varepsilon)}, \qquad \forall \alpha \in \N^d, \varepsilon \in (0,\varepsilon_2). \]
 Let $\beta \in \N^d, \beta \neq 0$. Applying (\ref{Landau-Kolmogorov inequality}) with $k = |\beta|$ and $n = 2|\beta|$, one obtains
 \begin{align*}
\sup_{x \in K} |f^{(\beta)}_\varepsilon(x)| &\leq \| g_\varepsilon^{(\beta)}\|_{L^\infty(\R^d)} \\
& \leq 2\pi d^{|\beta|}\| g_\varepsilon \|^{1/2}_{L^\infty(\R^d)}\sup_{|\alpha| = 2|\beta|} \| g_\varepsilon^{(\alpha)}\|^{1/2}_{L^\infty(\R^d)} \\
& \leq 2\pi C_1^{1/2}C_2^{1/2} \| \psi\|_{L^\infty(\R^d)}^{1/2} (dh_1)^{|\beta|} M_{2|\beta|}^{1/2} e^{(M(k_1/\varepsilon) - M(k_2/\varepsilon) )/2},
\end{align*}
for all $\varepsilon \in (0,\min(\varepsilon_1,\varepsilon_2))$. That $(f_\varepsilon)_\varepsilon \in \Ns(\Omega)$ then follows from \eqref{assM.2} and the inequality $M_{2p}^{1/2} \leq A^{1/2}H^p M_p$, $p \in \N$, which is a consequence of $(M.2)$.

We now show (\ref{Landau-Kolmogorov inequality}). We give a proof based on results on bounds for directional derivatives from \cite{Chen} and the one-dimensional Landau-Kolmogorov inequality itself. Denote by $\partial \slash \partial \xi$ the directional derivative in the direction $\xi$, where $\xi\in\mathbb{R}^{d}$ is a unit vector. It is shown in \cite{Chen}  that  
\[\sup_{|\alpha| = k} \| f^{(\alpha)}\|_{L^\infty(\R^d)} \leq  \sup_{|\xi| = 1} \left \| \frac{\partial^k f}{\partial^k \xi} \right \|_{L^\infty(\R^d)}.\]
 Let  $l(x, \xi)$ denote the line in $\R^d$ with direction $\xi$ passing through the point $x$. The one-dimensional Landau-Kolgomorov \cite{Kolmogorov} inequality yields
 \begin{align*}
 \left \| \frac{\partial^k f}{\partial^k \xi} \right \|_{L^\infty(\R^d)} &= \sup_{x \in \R^d}  \left \| \frac{\partial^k f}{\partial^k \xi} \right \|_{L^\infty(l(x,\xi))} 
 \\
 & \leq 2\pi \| f \|^{1 - k/n}_{L^\infty(\R^d)} \left \| \frac{\partial^n f}{\partial^n \xi} \right \|^{k/n}_{L^\infty(\R^d)}.
\end{align*}
Combining these two inequalities, we obtain
 \begin{align*}
\sup_{|\alpha| = k} \| f^{(\alpha)}\|_{L^\infty(\R^d)} & \leq 2\pi \| f \|^{1 - k/n}_{L^\infty(\R^d)} \sup_{|\xi| = 1} \left \| \frac{\partial^n f}{\partial^n \xi} \right \|^{k/n}_{L^\infty(\R^d)} \\
 & = 2\pi \| f \|^{1 - k/n}_{L^\infty(\R^d)} \sup_{|\xi| = 1}\sup_{x \in \R^d} \left | \sum_{j_1 = 1}^d \cdots \sum_{j_n = 1}^d  \frac{\partial^n f(x)}{\partial x_{j_1}\cdots \partial x_{j_n} } \xi_{j_1} \cdots \xi_{j_d}  \right |^{k/n} \\
&\leq 2\pi d^k \| f \|^{1 - k/n}_{L^\infty(\R^d)}\sup_{|\alpha| = n} \| f^{(\alpha)}\|^{k/n}_{L^\infty(\R^d)}.
\end{align*}

 \end{proof}
 
 We now discuss sheaf properties of $\Gss(\Omega)$. Given an open subset $\Omega'$ of  $\Omega$ and $f = [(f_\varepsilon)_\varepsilon] \in \Gss(\Omega)$, the restriction of $f$ to $\Omega'$ is defined as
\[f_{|\Omega'} = [(f_{\varepsilon|\Omega'})_\varepsilon] \in \Gss(\Omega'). \]

Using the existence of $\mathcal{D}^{\ast}$-partitions of the unity, ensured by $(M.3)'$ and $(M.1)$, one can show that the assignment $\Omega' \rightarrow \Gss(\Omega')$ is a fine sheaf of differential algebras on $\Omega$. In fact, this is not so hard to verify directly, but we mention the proof can be considerably simplified by reasoning exactly as in \cite[Thrm. 1.2.4]{GGK} with the aid of Proposition \ref{nullchar}. Furthermore, another application of Proposition \ref{nullchar} allows one to show that the sheaf $\mathcal{G}^{\ast}$ is also supple via a straightforward modification of the arguments given in \cite{Ober-Pil-Scarp}; we leave the details of such modifications to the reader. The support of $f \in \Gss(\Omega)$, a section of $\Gss$ on $\Omega$, is defined in the standard way. We recall \cite{Bengel-Schapira} that suppleness means that if $f\in\mathcal{G}^{\ast}(\Omega)$ has support in the union of two closed sets $Z_1\cup Z_2$, then it can be written as $f=f_{1}+f_{2}$, where each $f_{j}$ has support in $Z_{j}$. Either suppleness or fineness of a sheaf over a paracompact space implies softness (see \cite{Bengel-Schapira} for supple sheaves, the implication for fine sheaves is well known); hence $\mathcal{G}^{\ast}$ is a soft sheaf. We collect all this useful information in the following proposition.

\begin{proposition}\label{sheaf}
The functor $\Omega \rightarrow \Gss(\Omega)$ is a fine and supple sheaf of differential algebras. Every $\ast$-ultradifferential operator $P(D): \Gss \rightarrow \Gss$ is a sheaf morphism. The sheaf $\mathcal{G}^{(M_p)}$ is never flabby.
\end{proposition}
\begin{proof}
It only remains to show that $\mathcal{G}^{(M_p)}$ is not flabby. If it were flabby on $\Omega$, the restriction mappings $\mathcal{G}^{(M_p)}(\Omega)\to \mathcal{G}^{(M_p)}(\Omega')$ would be surjective for all $\Omega'\subset\Omega$. However, if $\partial\Omega'\cap \Omega\neq\emptyset$, this is not the case. In fact, let $x_0\in\partial\Omega'\cap \Omega$. 
In view of (\ref{assM.2}), the generalized function $[(f_{\varepsilon})_{\varepsilon}]$ with representative $f_{\varepsilon}(x)=e^{|x-x_0|^{-2} M(1/\varepsilon)}$ belongs to ${\mathcal{G}}^{(M_p)}(\Omega')$, but it has no extension to any neighborhood of $x_{0}$.  \end{proof}

We denote
$$
\Gss_c(\Omega)=\{f\in\Gss(\Omega):\: \operatorname*{supp} f \mbox{ is compact}\},
$$
the ideal of compactly supported $\ast$-generalized functions on $\Omega$. Note that $f \in \Gss_{c}(\Omega)$  if and only if there exists a representative $(f_\varepsilon)_\varepsilon$ of $f$ and a  $K \Subset \Omega$ such that $\operatorname{supp} f_\varepsilon \subseteq K$ for all $\varepsilon \in (0,1]$.
In such a case, we call the representative net \emph{compactly supported}.  

\subsection{Projective description of $\mathcal{G}^{\{M_p\}}(\Omega)$}\label{subsection projective}
 We now give an alternative projective type description of the algebra  of generalized functions of Roumieu type. As in the theory of ultradifferentiable functions \cite{Carmichael,Komatsu3, Pil94}, we shall do this by using the family $\mathcal{R}$ of all non-decreasing sequences $(r_j)_{j \in \N}$ tending to infinity. This set is partially ordered and directed by the relation $r_j \preceq s_j$, namely, there is a $j_0 \in \N$ such that $r_j \leq s_j$ for all $j \geq j_0$. Let $r_j \in \mathcal{R}$. We denote the associated function of the sequence $M_p\prod_{j = 0}^pr_j$ by $M_{r_j}$. Furthermore, for $K \Subset \Omega$, we define
 \begin{equation}
\label{projective seminorms}
  \| \phi \|_{K, r_j} := \sup_{\substack{x \in K \\ \alpha \in \N^d}} \frac{|\phi^{(\alpha)}(x)|}{M_{|\alpha|}\prod_{j = 0}^{|\alpha|}r_j}\:, \qquad \phi \in  C^{\infty}(\Omega). 
\end{equation}
 It is clear that the above seminorm is finite whenever $\phi \in \mathcal{E}^{\{M_p\}}(\Omega)$. 
We can define projective type spaces of moderate and negligible nets in terms of these norms as follows
 \begin{align*}
\widetilde{\mathcal{E}}^{\{M_p\}}_{\mathcal{M}}(\Omega) = \{ (f_\varepsilon)_\varepsilon \in  {\mathcal{E}^{\{M_p\}}(\Omega)}^{(0,1]} \, : \, &(\forall K \Subset \Omega) (\forall r_j \in \mathcal{R}) (\exists s_j \in \mathcal{R}) \\
 &\|f_\varepsilon\|_{K, r_j}  = O(e^{M_{s_j}(1/\varepsilon)})\},
\end{align*}
and
\begin{align*}
\widetilde{\mathcal{E}}^{\{M_p\}}_{\mathcal{N}}(\Omega) = \{ (f_\varepsilon)_\varepsilon \in  {\mathcal{E}^{\{M_p\}}(\Omega)}^{(0,1]} \, : \, &(\forall K \Subset \Omega) (\forall r_j \in \mathcal{R}) (\forall s_j \in \mathcal{R}) \\
 &\|f_\varepsilon\|_{K, r_j}  = O(e^{-M_{s_j}(1/\varepsilon)})\}.
\end{align*}
The goal of this subsection is to show the following result:
\begin{proposition}\label{projective}
We have $\mathcal{E}^{\{M_p\}}_{\mathcal{M}}(\Omega) =  \widetilde{\mathcal{E}}^{\{M_p\}}_{\mathcal{M}}(\Omega)$ and $\mathcal{E}^{\{M_p\}}_{\mathcal{N}}(\Omega) =  \widetilde{\mathcal{E}}^{\{M_p\}}_{\mathcal{N}}(\Omega)$ .
\end{proposition}
The proof of Proposition \ref{projective} is based on the ensuing lemma, which only makes use of the conditions $(M.1)$ and $(M.2)$ on the weight sequence $M_p$ (actually $(M.1)$ and $(M.2)'$ suffice).
\begin{lemma}\label{projective1}
Let $g: [a,\infty) \rightarrow [0,\infty)$ for some $a \in \R$. Then,
\begin{enumerate}
\item[$(i)$] $g(t) = O(e^{M(kt)})$ for all $k > 0$ if and only if $g(t) = O(e^{M_{r_j}(t)})$ for some sequence $r_j \in \mathcal{R}$. 
\item[$(ii)$] $g(t)= O(e^{-M_{r_j}(t)})$ for all sequences $r_j \in \mathcal{R}$ if and only if $g(t) = O(e^{-M(kt)})$ for some $k > 0$.  
\end{enumerate}
\end{lemma}
\begin{proof}
We only need to show the direct implications, the "if" parts are clear. We may assume $a = 0$.
\newline $(i)$ We first show that there exists a subordinate function $\rho: [0, \infty) \to [0, \infty)$ (which means that $\rho$ is continuous, increasing, and satisfies $\rho(0) = 0$ and $\rho(t) = o(t)$) such that $g(t) = O(e^{M(\rho(t))})$. Our assumption and \eqref{assM.2} imply that 
\[ (\forall k > 0)(\exists R > 0)(\forall t \geq R)(g(t) \leq Ae^{M(kt)}). \]
Hence we can inductively select a sequence $(t_n)_{n \in \Z_+}$ with $t_1 = 0$ that satisfies 
\[ g(t) \leq Ae^{M(t/(n+1))}, \quad \forall t \geq t_n, \qquad \frac{t_{n}}{n} \geq \frac{t_{n-1}}{n-1} + 1, \qquad \forall n \geq 2.  \]
Let $l_n$ denote the line through the points $(t_n, t_n/n)$ and $(t_{n+1}, t_{n+1}/(n+1))$, and define
\[ \rho(t) = l_n(t), \qquad \mbox{for } t \in [t_n, t_{n+1}), \]
The function $\rho$ is subordinate and, moreover, $g(t) \leq Ae^{M(\rho(t))}$ for all $t \geq t_2$. Thus, it suffices to prove that there is a sequence $r_j \in \mathcal{R}$ such that $M(\rho(t)) \leq M_{r_j}(t) + C$ for $t \geq 0$, for some $C > 0$. Applying \cite[Lemm. 3.12]{Komatsu}, one finds a weight sequence $N_p$ satisfying $(M.1)$ with associated function $N$ such that $M_p \prec N_p$ and $M(\rho(t)) \leq N(t)$ for $t \geq 0$. On the other hand, \cite[Lemm. 3.4]{Komatsu3} yields the existence of a sequence $r_j \in \mathcal{R}$ with  $M_p\prod_{j=1}^p r_j \prec N_p$. The result now follows from  \cite[Lemm. 3.10]{Komatsu}.
\\ \\
$(ii)$ The last part of the proof of $(i)$ and our assumption imply that $g(t) = O(e^{-M(\rho(t))})$ for all subordinate functions $\rho$. Suppose that $g(t) = O(e^{-M(kt)})$ does not hold for any $k > 0$. 
We could therefore find a sequence $(t_n)_{n \in \Z_+}$ with $t_1 = 0$ and 
\[ g(t_{n}) e^{M(t_n/n)}\geq n, \qquad \frac{t_{n}}{n} \geq \frac{t_{n-1}}{n-1} + 1, \qquad \forall n \geq 2.  \]
Exactly as in the proof of $(i)$, we define the subordinate function $\rho(t)$ given by  line through $(t_n, t_n/n)$ and $(t_{n+1}, t_{n+1}/(n+1))$ for $t \in [t_n, t_{n+1})$.
We would then have
\[g(t_n)e^{M(\rho(t_n))} = g(t_n)e^{M(t_n/n)} \geq n, \qquad \forall n \geq t_2,\]
contradicting $g(t) = O(e^{-M(\rho(t))})$.
\end{proof}

\begin{proof}[Proof of Proposition \ref{projective}.] We start by showing $\mathcal{E}^{\{M_p\}}_{\mathcal{M}}(\Omega) =  \widetilde{\mathcal{E}}^{\{M_p\}}_{\mathcal{M}}(\Omega)$. Let $(f_\varepsilon)_\varepsilon \in \mathcal{E}^{\{M_p\}}_{\mathcal{M}}(\Omega) $ and fix $K \Subset \Omega$ and $r_j  \in \mathcal{R}$. Since the seminorms (\ref{projective seminorms}) are continuous on $\mathcal{E}^{\{M_p\},h}(\Omega)$ for each $h >0$, we obtain $\|f_\varepsilon \|_{K, r_j} = O(e^{M(k/\varepsilon)})$ for all $k > 0$. By applying Lemma \ref{projective1}$(i)$ to the function $g(t) = \|f_{1/t} \|_{K, r_j}$, $t \geq 1$, we find a sequence $s_j \in \mathcal{R}$ such that $\|f_\varepsilon \|_{K, r_j} = O(e^{M_{s_j}(1/\varepsilon)})$.
Conversely, let $(f_\varepsilon)_\varepsilon \in \widetilde{\mathcal{E}}^{\{M_p\}}_{\mathcal{M}}(\Omega)$ and suppose that $(f_\varepsilon)_\varepsilon \notin \mathcal{E}^{\{M_p\}}_{\mathcal{M}}(\Omega) $. Then, there are $K \Subset \Omega$, $k > 0$, and a sequence $\varepsilon_n \searrow 0$ such that
\begin{equation}
\label{eqcontradiction1}
\| f_{\varepsilon_n}\|_{\mathcal{E}^{\{M_p\},n}(K)} \geq n e^{M(k/\varepsilon_n)}, \qquad \forall n \in \N. 
\end{equation}
Choose $K \Subset K' \Subset \Omega$ and $\psi \in \mathcal{D}^{\{M_p\}}(K')$ with $\psi \equiv 1$ in a neighborhood of $K$. In order to produce a contradiction to (\ref{eqcontradiction1}), we shall show that the sequence $(g_n)_n := (f_{\varepsilon_n}\psi e^{-M(k/\varepsilon_n)})_n$ is bounded in $\mathcal{D}^{\{M_p\}}(K')$. It well known \cite{Komatsu3} that 
\[ \mathcal{D}^{\{M_p\}}(K') = \varprojlim_{r_j \in \mathcal{R}} \mathcal{D}^{\{M_p\},r_j}(K'), \]
as t.v.s.\,, where $\mathcal{D}^{\{M_p\},r_j}(K')$ denotes the Banach space of all $\phi \in \mathcal{D}(K')$ with $\| \phi \|_{K', r_j} < \infty$. Hence we must prove that $(g_n)_n$ is bounded in $\mathcal{D}^{\{M_p\},r_j}(K')$ for each $r_j \in \mathcal{R}$. But $(f_\varepsilon)_\varepsilon \in \widetilde{\mathcal{E}}^{\{M_p\}}_{\mathcal{M}}(\Omega)$, thus we can find $s_j \in \mathcal{R}$ and $C > 0$ such that
\[ \sup_{n \in \N}\| g_n \|_{K', r_j} \leq C\sup_{n \in \N}e^{M_{s_j}(1/\varepsilon_n) - M(k/\varepsilon_n)} < \infty. \]
 For the second equality, the inclusion $\mathcal{E}^{\{M_p\}} _{\mathcal{N}} (\Omega) \subseteq  \widetilde{\mathcal{E}}^{\{M_p\}}_{\mathcal{N}} (\Omega)$ is clear, whereas the converse inclusion is a direct consequence of Proposition \ref{nullchar} and  Lemma \ref{projective1}$(ii)$.
\end{proof}

\subsection{Generalized point values.}
In this subsection we introduce the ring of $\ast$-generalized numbers in order to regard $\ast$-generalized functions as pointwise defined objects. The generalized numbers  $\widetilde{\C}^\ast$ are introduced as follows. 
Define 
\[ \mathbb{C}^{(M_p)}_{\mathcal{M}}= \{ (z_\varepsilon)_\varepsilon \in \C^{(0,1]} \, : \,  (\exists k > 0) (|z_\varepsilon| = O(e^{M(k/\varepsilon)})) \},  \]
\[ \mathbb{C}^{\{M_p\}}_{\mathcal{M}} = \{ (z_\varepsilon)_\varepsilon \in \C^{(0,1]} \, : \,  (\forall k > 0) (|z_\varepsilon| = O(e^{M(k/\varepsilon)})) \},  \]
and the ideals
\[ \mathbb{C}^{(M_p)}_{\mathcal{N}} = \{ (z_\varepsilon)_\varepsilon \in \C^{(0,1]} \, : \,  (\forall k > 0) (|z_\varepsilon| = O(e^{-M(k/\varepsilon)})) \},  \]
\[  
\mathbb{C}^{\{M_p\}}_{\mathcal{N}}= \{ (z_\varepsilon)_\varepsilon \in \C^{(0,1]} \, : \,  (\exists k > 0) (|z_\varepsilon| = O(e^{-M(k/\varepsilon)})) \}. 
 \]
Note that Lemma \ref{projective1} yields alternative descriptions of these rings in the Roumieu case:
\[ \mathbb{C}^{\{M_p\}}_{\mathcal{M}} = \{ (z_\varepsilon)_\varepsilon \in \C^{(0,1]} \, : \,  (\exists s_j \in \mathcal{R}) (|z_\varepsilon| = O(e^{M_{s_j}(1/\varepsilon)})) \},  \]
and 
\[ \mathbb{C}^{\{M_p\}}_{\mathcal{N}} = \{ (z_\varepsilon)_\varepsilon \in \C^{(0,1]} \, : \,  (\forall s_j \in \mathcal{R}) (|z_\varepsilon| = O(e^{-M_{s_j}(1/\varepsilon)})) \}. 
\]

The ring of $\ast$-\emph{generalized numbers} is then defined as the factor ring
\[ \widetilde{\C}^\ast = \mathbb{C}^{\ast}_{\mathcal{M}}/ \mathbb{C}^{\ast}_{\mathcal{N}}. \]
The reader should be aware of the fact that $\widetilde{\C}^\ast$ is not a field (this can be shown exactly with the same examples used for the rings of Colombeau generalized numbers, see e.g. \cite[Ex. 1.2.33, p. 32]{GGK}). Furthermore, the elements of $\mathbb{C}$ are canonically embedded into $\widetilde{\C}^\ast$ as constant nets. Likewise, one can define the subring $\widetilde{\R}^\ast$ and the modules $\widetilde{{\R}^{d}}^\ast$ and $\widetilde{{\C}^{d}}^\ast$.  More generally, we can associate to every open set $A\subseteq \mathbb{R}^{d}$ a set of \emph{$\ast$-generalized points} $\widetilde{A}^\ast\subset\widetilde{{\R}^{d}}^\ast$ in the following way. Let $A^\ast_{\mathcal{M}} = \{ (x_\varepsilon)_\varepsilon \in A^{(0,1]} \, : \, (|x_\varepsilon|)_\varepsilon \in \mathbb{C}^{\ast}_{\mathcal{M}} \} $;
two elements $(x_\varepsilon)_\varepsilon,  (y_\varepsilon)_\varepsilon \in A^\ast_{\mathcal{M}}$ are said to be equivalent, denoted, say, by $(x_\varepsilon)_\varepsilon \sim  (y_\varepsilon)_\varepsilon$, if $(|x_\varepsilon - y_\varepsilon|)_\varepsilon \in \mathbb{C}^{\ast}_{\mathcal{N}}$. We define then $ \widetilde{A}^\ast =  A^{\ast}_{\mathcal{M}} / \sim$. 

We are ready to discuss some pointwise properties of $\ast$-generalized functions. Let $f = [(f_\varepsilon)_\varepsilon] \in \Gss(\Omega)$. If $x \in \Omega$, then it is obvious that $ f(x):=[(f_\varepsilon(x))_\varepsilon]$ is a well defined element of $\widetilde{\C}^\ast$, the point value of $f$ at $x$. Thus, every $\ast$-generalized function induces an actual $\widetilde{\C}^\ast$-valued function  on $\Omega$, provided by $x \mapsto f(x)$. However, just as in the case of Colombeau generalized functions, $f$ is not determined by this mapping on $\Omega$, in the sense that there are non-identically zero generalized functions with $f(x)=0$ for all $x\in\Omega$. In fact, the same well-known examples from classical Colombeau theory work also in our setting to prove this assertion (see e.g. \cite[Ex. 1.2.43 (iv)]{GGK}). This problem can be overcome by introducing point values of generalized functions also at generalized points. Naturally, in general $f(x)$ cannot be defined for arbitrary $x\in\widetilde{\Omega}^{\ast}$ because $(f_{\varepsilon}(x_{\varepsilon}))_{\varepsilon}$ may not even be an element of 
$\mathbb{C}^{\ast}_{\mathcal{M}}$ for arbitrary $(x_{\varepsilon})_{\varepsilon}\in\Omega^{\ast}_{\mathcal{M}}$. On the other hand, $(f_{\varepsilon}(x_{\varepsilon}))_{\varepsilon}\in\mathbb{C}^{\ast}_{\mathcal{M}}$ if the net $(x_{\varepsilon})_{\varepsilon}$ belongs to the so-called subset of compactly supported points of $\widetilde{\Omega}^\ast$, that is,
\[ \widetilde{\Omega}^\ast_c = \{ x = [(x_\varepsilon)_\varepsilon] \in \widetilde{\Omega}^\ast \, : \, (\exists K \Subset \Omega)(\exists \varepsilon_0 > 0)(\forall \varepsilon \in (0, \varepsilon_0))(x_\varepsilon \in K) \}. \]
We can therefore define the point value of $f\in \Gss(\Omega)$ at $x\in\widetilde{\Omega}^\ast_c$ as $f(x):=[(f_\varepsilon(x))_\varepsilon]$; obviously, the point value $f(x)$ does not depend on the representative of $f$, neither does it on the representative of $x$, as the mean value theorem implies. The next proposition is a pointwise characterization of $\ast$-generalized functions; it shows that every $f\in\mathcal{G}^{\ast}(\Omega)$ can be associated with the pointwise defined mapping $f:\widetilde{\Omega}_{c}^\ast\to\widetilde{\mathbb{C}}^{\ast}$ in a  one-to-one fashion.
 \begin{proposition}\label{pointwise} One has $f = 0$ in $\Gss(\Omega)$ if and only if $f(x) = 0$ in $\widetilde{\C}^\ast$ for all $x \in \widetilde{\Omega}^\ast_c$.
 \end{proposition}
 \begin{proof}
The direct implication is obvious. Conversely, suppose that $f=[(f_{\varepsilon})_{\varepsilon}]$ satisfies $f(x) = 0$ for all $x \in \widetilde{\Omega}^\ast_c$. Suppose $f \neq 0$ in $\Gss(\Omega)$. By Proposition \ref{nullchar}, one could find $K \Subset \Omega$, $k > 0$, and sequences  $\varepsilon_n \searrow 0$ and $x_n \in K$ ($K \Subset \Omega$ and sequences  $\varepsilon_n \searrow 0$ and $x_n \in K$) such that
 \[ |f_{\varepsilon_n}(x_n)| \geq e^{-M(k/\varepsilon_n)} \qquad (|f_{\varepsilon_n}(x_n)| \geq e^{-M(1/(n\varepsilon_n))}), \qquad \forall n \in \N. \]
 For $\varepsilon > 0$, we set $x_\varepsilon = x_k$ if $\varepsilon_{k+1} < \varepsilon \leq \varepsilon_k$. Hence $x = [(x_\varepsilon)_\varepsilon] \in \widetilde{\Omega}^\ast_c$ and, by construction, $f(x) \neq 0$ in $\widetilde{\C}^\ast$.
  \end{proof}

\section{Embedding of $M_p$-ultradistributions} \label{section-embedding}
 We shall now embed ${\Ds}'(\Omega)$ into $\Gss(\Omega)$ in such a way that the multiplication of functions from $\Es(\Omega)$ is preserved. We shall do so in several steps. We first construct an embedding of ${\Es}'(\Omega)$ into the ideal $\Gss_{c}(\Omega)$ by means of convolution with a suitable mollifier and then use the sheaf-theoretic properties of $\Gss(\Omega)$ (Proposition \ref{sheaf}) to extend the embedding to the whole space ${\Ds}'(\Omega)$.
  
We prepare the ground with a discussion about the properties of the mollifiers to be employed in our embedding of ${\Es}'(\Omega)$ into $\Gss_{c}(\Omega)$. Let $N_p$ be a weight sequence with associated function $N$. We always assume that $N_p$ fulfills $(M.1)$ and $(M.3)'$.  We fix an even function $\phi \in \Ft^{-1}(\mathcal{D}^{(N_p)}(\R^d))$ with the following properties
\[
\widehat{\phi}(x) = 1, \quad \mbox{for }|x| \leq 1, \qquad \widehat{\phi}(x) = 0, \quad \mbox{for } |x| \geq 2.
 \]
The net
\[\phi_\varepsilon(x) = \frac{1}{\varepsilon^d}\phi\left(\frac{x}{\varepsilon} \right), \qquad \varepsilon \in (0,1],\]
is called a $(N_p)$-\emph{net of mollifiers}. Besides $(M.1)$ and $(M.3)'$, we also impose the following assumption on $N_p$ relating the growths of $M$ and $N$: For each $l > 0$
\begin{equation}
\label{eqass3N}
 2M(t) \leq N(l t) + C, \qquad \forall t \geq 0,
 \end{equation}
 for some $C=C_{l} > 0$. The existence of such sequences is ensured by some results by Roumieu. More precisely,
\begin{lemma}\label{Roum}
It is always possible to select a weight sequence $N_p$ with the properties $(M.1)$, $(M.3)'$, and $(\ref{eqass3N})$.
\end{lemma}

\begin{proof}
Set $ M_p' = \min_{q\leq p} M_qM_{p-q}$. A lemma by Roumieu \cite[Lemm. 3.5]{Komatsu} tells that $M'_p$ has associated function $2M$ and in fact satisfies $(M.1)$.  It is clear that $(M'_p)$ satisfies $(M.3)'$ (cf. \eqref{assM.3}). Furthermore, another result by Roumieu \cite[p.66]{Roumieu} yields the existence of a weight sequence $N_p$ satisfying $(M.1)$, $(M.3)'$, and $N_p \prec M'_p$; \cite[Lemma 3.10]{Komatsu} then gives (\ref{eqass3N}).
\end{proof}

We are ready to embed ${\Es}'(\Omega)$ into $\Gss_c(\Omega)$. First notice that $\phi_{\varepsilon}$ is an entire function, thus, the convolution $f\ast \phi_{\varepsilon}$ is well defined for $f\in {\Es}'(\Omega)$.

\begin{proposition}\label{embedding} The mapping
\begin{equation}
\label{eq embedding compact}
\iota_c: {\Es}'(\Omega) \rightarrow \Gss_{c}(\Omega):\quad f \mapsto \iota_c(f) = [((f \ast \phi_\varepsilon)_{|\Omega})_\varepsilon],
\end{equation}
is a linear embedding. Furthermore, $\iota_{c|\Ds(\Omega)} = \sigma$, where $\sigma$ is the constant embedding  
$(\ref{eqconstantembedding})$.
\end{proposition}

\begin{proof} 
Let $f \in {\Es}'(\Omega)$. We first need to check that $(f\ast\phi_{\varepsilon})_{\varepsilon}\in \mathcal{E}^{\ast}_{\mathcal{M}}$.  Let $K \Subset  \Omega$ and $h > 0$ be arbitrary. By continuity of $f$  (cf. the proof of \cite[Thrm. 6.10]{Komatsu}), we find
\[(\exists K' \Subset \Omega)( \exists k > 0)(\exists C > 0) \qquad ((\exists K' \Subset \Omega)( \forall k > 0)(\exists C > 0))    \]
\[ \| f \ast \phi_\varepsilon \|_{\Eh(K)} \leq C\| \phi_\varepsilon \|_{\El(K-K')}\: , \]
where $ l = \min(h,k)/ H$.
Employing \cite[Lemm. 3.3]{Komatsu}, we have
\begin{align*}
\| \phi_\varepsilon \|_{\El(K-K')} &\leq \frac{1}{(2\pi)^d}  \int_{\R^d} |\widehat{\phi}(\varepsilon \xi)|e^{M(|\xi|/l)} \dxi \\
& \leq \frac{1}{(2\pi\varepsilon)^d} \int_{\R^d} |\widehat{\phi}(t)|e^{M(|t|/(l\varepsilon))} \dt \\
&\leq \frac{1}{(2\pi\varepsilon)^d} \| \widehat{\phi}\|_{L^1}\:e^{M(2/(l\varepsilon))}.
\end{align*}
This proves that $\iota_{c}:{\mathcal{E}^{\ast}}'(\Omega)\to \mathcal{G}^{\ast}(\Omega)$ is a well-defined linear mapping. Suppose that $((f \ast \phi_\varepsilon)_{|\Omega})_\varepsilon \in \Ns(\Omega)$. Given $\psi \in  \Ds(\Omega)$, we have
$\langle f, \psi \rangle = \lim_{\varepsilon \to 0^+} \langle f \ast \phi_\varepsilon, \psi \rangle = 0$, which shows that $\iota_{c}$ is injective. Finally, since $\mathcal{G}^{\ast}$ is a sheaf and, as we have just shown,  $\iota_{c}:\mathcal{E}^{\ast}(\Omega')\to \mathcal{G}^{\ast}(\Omega')$ is injective on any $\Omega'\subseteq \Omega$, we obtain $\operatorname*{supp}f \subseteq \operatorname*{supp}\iota_{c}( f)$. It follows that  $\iota(\mathcal{E}^{\ast}(\Omega))\subset \mathcal{G}^{\ast}_{c}(\Omega)$, as claimed in (\ref{eq embedding compact}). 

Let now $f \in \Dh(K)$. Using \cite[Lemm. 3.3]{Komatsu} once again,
\begin{align*}
|(f\ast \phi_\varepsilon)(x) - f(x)| &\leq |\Ft^{-1}(\widehat{f}\widehat{\phi_\varepsilon}  - \widehat{f})(x)| 
\leq \frac{1}{(2\pi)^d}\int_{\R^d} |\widehat{f}(\xi)||\widehat{\phi}(\varepsilon \xi) - 1| \dxi \\
&\leq \frac{|K|}{(2\pi)^d} \|f\|_{\Eh(K)} (\|\widehat{\phi}\|_{L^\infty} + 1) \int_{|\xi|\geq 1/\varepsilon} e^{-M(|\xi| / (\sqrt{d}h))} \dxi,
\end{align*}
where $|K|$ denotes the Lebesgue measure of $K$. Thus, 
\begin{equation}
 |(f\ast \phi_\varepsilon) (x)- f(x)| \leq C\|f\|_{\Eh(K)}e^{-M(1/ (\sqrt{d}hH\varepsilon))}  
\label{ongelijkheid}
\end{equation}
where
\[C = \frac{A|K|}{(2\pi)^d}   (\|\widehat{\phi}\|_{L^\infty} + 1)\int_{\R^d} e^{-M(|\xi| / (\sqrt{d}hH))} \dxi < \infty,\]
and $A,H$ are the constants occurring in \eqref{assM.2}.
That $\iota_{c}(f)=\sigma(f)$ now follows from Proposition \ref{nullchar}.
\end{proof}

\begin{corollary}\label{coro} 
Let $f \in {\Es}'(\Omega)$, $(f_\varepsilon)_\varepsilon$ a representative of $\iota_c(f)$, and $K \Subset \Omega$. Then
\[(\exists h > 0)(\exists C > 0)( \exists \varepsilon_0 > 0)(\forall \psi \in \Df(K))( \forall \varepsilon < \varepsilon_0)\]
\[| \langle f -f_\varepsilon, \psi \rangle | \leq C\| \psi\|_{\Eh(K)}e^{-M(1/ (\sqrt{d} H^2h\varepsilon))}.\]
\end{corollary} 
\begin{proof}
Condition $(M.2)$ implies that for all $h>0$
\[ \| \psi^{(\alpha)}\|_{\mathcal{E}^{\{M_p\},Hh}(K)} \leq AM_{|\alpha|}(Hh)^{|\alpha|}\|\psi\|_{\mathcal{E}^{\{M_p\},h}(K)}, \qquad \forall \psi \in \Df(K).\]
Hence, by applying inequality \eqref{ongelijkheid} to $\psi^{(\alpha)}$ instead of $f$, we find that
\[\sup_{\substack{x \in \R^d \\ \alpha \in \N^d}} \frac{|(\psi \ast \phi_\varepsilon)^{(\alpha)}(x) - \psi^{(\alpha)}(x)|}{(Hh)^{|\alpha|}M_{|\alpha|}} \leq AC \|\psi\|_{\Eh(K)}e^{-M(1/( \sqrt{d}H^2h\varepsilon))}.\]
The rest follows from the continuity of $f$ and the definition of the ideal $\Ns(\Omega)$.
\end{proof}

So far we have only used the assumption $(M.3)'$ on $N_p$. The other two assumptions become now crucial to show that the embedding (\ref{eq embedding compact}) is support preserving. We need an auxiliary lemma in order to establish this fact (which only makes use of the assumptions $(M.1)$ and $(M.3)'$ on $N_p$).
   
\begin{lemma}\label{PWS} Let $K$ be a compact set of $\R^d$. Then, for each $h >0$ there is $C > 0$ such that

\[ |(\widehat{\psi})^{(\alpha)}(\xi)| \leq C(2\max\{|K|,1\})^{|\alpha|}\|\psi \|_{\mathcal{E}^{\{N_p\},h}(K)} e^{-N(|\xi| / (2\sqrt{d}h))}, \qquad \forall \xi \in \R^d,\]
for all $\psi \in \mathcal{D}^{\{N_p\},h}(K)$ and $\alpha \in \N^d$.
\end{lemma}
\begin{proof}
By \cite[Lemm. 3.3]{Komatsu}, we have that
\[|(\widehat{\psi})^{(\alpha)}(\xi)| = |\Ft(x^\alpha\psi(x))(\xi)| \leq |K|\|x^\alpha\psi \|_{\mathcal{E}^{\{N_p\},2h}(K)} e^{-N(|\xi| / (2\sqrt{d}h))}, \qquad \forall \xi \in \R^d.\]
Using \cite[Prop 2.7]{Komatsu},
\[\|x^\alpha\psi \|_{\mathcal{E}^{\{N_p\},2h}(K)} \leq \|x^\alpha\|_{\mathcal{E}^{\{N_p\},h}(K)} \|\psi\|_{\mathcal{E}^{\{N_p\},h}(K)}.\]
Set $D=\max\{|K|,1\}$. Since $p! \prec N_p$ (cf. \cite[Lemm. 4.1]{Komatsu}),
\[
\|x^\alpha\|_{\mathcal{E}^{\{N_p\},h}(K)} = \sup_{\substack{x \in K \\ \beta \leq \alpha}} \binom{\alpha}{\beta} \frac{\beta!|x^{\alpha-\beta}|}{h^{|\beta|}N_{|\beta|}}
\leq (2D)^{|\alpha|}\sup_{\beta \leq \alpha} \frac{|\beta|!}{h^{|\beta|}N_{|\beta|}}
\leq C(2D)^{|\alpha|},
\]
for some $C$, depending only on $h$ and $N_p$.
\end{proof}

We then have that the support of $f\in{\Es}'(\Omega)$ coincides with that of $\iota_c(f)$. 

\begin{proposition}\label{support} Let $f \in {\Es}'(\Omega)$. Then, $\operatorname*{supp}\iota_{c}(f) = \operatorname*{supp} f$. 
\end{proposition}
\begin{proof} We have already proved the inclusion $\operatorname{supp}f \subseteq\operatorname{supp}\iota_{c}( f)$ within the proof of Proposition \ref{embedding}.
So, we should show here that  $\operatorname*{supp}\iota_{c}(f) \subseteq \operatorname{supp} f$. Let $K \Subset \Omega \backslash \operatorname*{supp} f$. Choose $K'\Subset\Omega$ such that a neighborhood of $\operatorname*{supp} f$ is contained in $K'$ and $K'\cap K=\emptyset$. Set $\delta:= d(K,K') > 0$.  By the continuity of $f$, we can find $k>0$ and $C>0$ (for each $k>0$ there is $C>0$) such that 
\[
|(f \ast \phi_\varepsilon)(x)| \leq C \|\phi_\varepsilon \|_{\Ek(K-K')}\: , \quad \forall x\in K.
\]
Set $\psi = \widehat{\phi} \in \mathcal{D}^{(N_p)}(\R^d)$, so that $\widehat{\psi}(\xi) = (2\pi)^d\phi(\xi)$. Lemma \ref{PWS} implies that 
\begin{align*}
 \| \phi_\varepsilon \|_{\Ek(K-K')} &= \frac{1}{(2\pi\varepsilon)^d}\sup_{\substack{u \in K- K' \\ \alpha \in \N^d}} \frac{|\widehat{\psi}^{(\alpha)}(-u/\varepsilon)|}{(k\varepsilon)^{|\alpha|}M_{|\alpha|}} \\
&\leq \frac{C'}{(2\pi\varepsilon)^d}\|\psi \|_{\mathcal{E}^{\{N_p\},k/2}(\bar{B}(0,2))}e^{-N(\delta / (\sqrt{d}k\varepsilon)) + M(D/(k\varepsilon))}
\\
&\leq C''e^{- M(D/(k\varepsilon))},
\end{align*}
in view of (\ref{eqass3N}), for some $C',C''$ and $D> 0$. We then obtain that $(\iota_{c}(f))_{|(\Omega \backslash \operatorname*{supp} f)}=0$ from Proposition \ref{nullchar}.
\end{proof}

We have all ingredients we need to realize an embedding of ${\mathcal{D}^{\ast}}'(\Omega)$ into $\mathcal{G}^{\ast}(\Omega)$. It is a standard exercise to show that any support preserving linear embedding between the compactly supported global sections of two fine sheaves on a Hausdorff second countable locally compact topological space can be uniquely extended to a sheaf embedding. In fact, the latter conclusion remains valid if one replaces fineness of the sheaves by the weaker hypothesis of softness \cite[Lemm. 2.3, p. 228]{Komatsu73b}. Summing up all of our results, we have shown:
\begin{theorem}\label{main}
 There is a linear embedding  $\iota = \iota_\Omega: {\Ds}'(\Omega) \rightarrow \Gss(\Omega)$ having the following properties,
\begin{enumerate}
\item[$(i)$] $\iota_{|{\Es}'(\Omega)} = \iota_c.$
\item[$(ii)$] $\iota$ commutes with $\ast$-ultradifferential operators, that is, for all $\ast$-ultradifferential operators $P(D)$
\[ P(D)\iota(f) = \iota(P(D)f), \quad f \in {\Ds}'(\Omega).\]
\item[$(iii)$] $\iota_{|{\Es}(\Omega)}$ coincides with the constant embedding $\sigma.$ Consequently, 
$$
\iota(fg) = \iota(f)\iota(g),\quad f,g \in \Es(\Omega).$$ 
\end{enumerate}
Moreover, the entirety of all $\iota_\Omega: {\Ds}'(\Omega) \rightarrow \Gss(\Omega)$ is a sheaf monomorphism ${\Ds}' \rightarrow \Gss$ on any open subset of $\mathbb{R}^{d}$.
\end{theorem}

It is clear that the embedding $\iota:{\mathcal{D}^{\ast}}'(\Omega)\to \Gss(\Omega)$ constructed in this section satisfies the properties $(P.1)$--$(P.3)$ stated in Section \ref{imposs section} with $\dagger = \ast$.  Hence, our embedding is optimal in the sense discussed there.

\section{Regular generalized functions of class $(M_p)$ and $\{M_p\}$.}\label {section regularity}
The goal of this section is to introduce a notion of regularity (with respect to $\ast$-ultradifferentiability) in the algebra of $\ast$-generalized functions. This will be done via the algebra ${\mathcal{G}^{\ast}}^{,\infty}(\Omega)$ defined below and will be exploited in the next section to study microlocal properties of $\ast$-generalized functions with respect to ${\mathcal{G}^{\ast}}^{,\infty}$-microregularity.  Note that the counterpart of our algebra in the context of classical Colombeau theory is Oberguggenberger's algebra $\mathcal{G}^{\infty}(\Omega)$ \cite{Ober}; see also \cite{hormann2004,PilScar,PilScar2001b,p-s-v2013a} for regularity analysis of generalized functions.

We define the algebra of \emph{regular} $\ast$-generalized functions on $\Omega$ as
\[ \Gsi(\Omega) = \mathcal{E}^{\ast,\infty}_{\mathcal{M}}(\Omega)/\Ns(\Omega),\]
where

\begin{align*}
\mathcal{E}^{(M_p),\infty}_{\mathcal{M}}(\Omega) = \{ (f_\varepsilon)_\varepsilon \in \Ef(\Omega)^{(0,1]} \, : \, &(\forall K \Subset \Omega) (\exists k > 0) (\forall h > 0) \\
 &\|f_\varepsilon\|_{\Eh(K)}  = O(e^{M(k/\varepsilon)})     \}
\end{align*}
and
\begin{align*}
\mathcal{E}^{\{M_p\},\infty}_{\mathcal{M}}(\Omega) = \{ (f_\varepsilon)_\varepsilon \in \Ee(\Omega)^{(0,1]} \, : \, &(\forall K \Subset \Omega) (\exists h > 0) (\forall k > 0) \\
 &\|f_\varepsilon\|_{\Eh(K)}  = O(e^{M(k/\varepsilon)})       \}.
\end{align*}
We remark that a similar argument to the one employed in the proof of  Proposition \ref{projective} leads to 
\begin{align*}
\mathcal{E}^{\{M_p\},\infty}_{\mathcal{M}}(\Omega) = \{ (f_\varepsilon)_\varepsilon \in \Ee(\Omega)^{(0,1]} \, : \, &(\forall K \Subset \Omega) (\exists s_j \in \mathcal{R})(\forall r_j \in \mathcal{R}) \\
 &\|f_\varepsilon\|_{K,r_j}  = O(e^{M_{s_j}(1/\varepsilon)})       \},
\end{align*}
we omit details. Clearly $\Gsi$ is a subsheaf of $\Gss$. It is also fine, supple, non-flabby, and closed under $\ast$-ultradifferential operators. A generalized function $f\in\Gss(\Omega)$ is said to be \emph{${\Gss}^{,\infty}$-regular} (in short: \emph{$\ast$-regular}) on $\Omega'\subseteq \Omega$ if its restriction to $\Omega'$ belongs to ${\Gss}^{,\infty}(\Omega')$. The \emph{$\ast$-singular support} of $f \in \Gss(\Omega)$, denoted as $\operatorname*{sing\: supp}_{g,\ast} f$, is defined as the complement in $\Omega$ of the largest open set on which $f$ is $\ast$-regular. We need the following simple but useful result, a Paley-Wiener-Komatsu type characterization of compactly supported regular $\ast$-generalized functions. It follows of course directly from the classical result \cite[Lemm. 3.3]{Komatsu}.
\begin{lemma}\label{Pwsalg}
Let  $f \in \Gss_c(\Omega)$ and let $(f_\varepsilon)_\varepsilon$  be a compactly supported representative of $f$. Then, $f  \in \Gsi (\Omega)$ if and only if 
\[( \exists k > 0)(\forall h > 0) \qquad (( \exists h > 0)(\forall k > 0))    \]
\[ \sup_{\xi \in \R^d}|\widehat{f}_\varepsilon(\xi)|e^{M(|\xi|/h)} = O(e^{M(k/\varepsilon)}). \]
\end{lemma}

The next regularity theorem is the main result of this section. It gives a precise characterization of the embedded image of $\Es(\Omega)$ under the embedding $\iota$ constructed in Section \ref{section-embedding} in terms of the algebra $\mathcal{G}^{\ast,\infty}(\Omega)$.

\begin{theorem}\label{mainreg} $\Gsi(\Omega) \cap \iota({\Ds}'(\Omega)) =  \iota(\Es(\Omega)).$
\end{theorem}

\begin{remark} Theorem \ref{mainreg} considerably improves  the characterizations from \cite{PilScar,PilScar2001b} of $\ast$-ultradistributions that belong to $\Es(\Omega)$ via the rate of growth of nets of regularizations. In fact, contrary to the quoted former results, where only sufficient conditions were obtained, Theorem \ref{mainreg} provides a necessary and sufficient condition for $\ast$-ultradifferentiability.
\end{remark}
\begin{proof}
The inclusion $ \iota(\Es(\Omega)) \subseteq \Gsi(\Omega) \cap  \iota({\Ds}'(\Omega))$ is clear. Conversely, let $f \in {\Ds}'(\Omega)$ such that $\iota(f) \in \Gsi(\Omega)$. We may assume without loss of generality that $f$ is compactly supported. By the Paley-Wiener-Komatsu theorem \cite[Thrm. 9.1]{Komatsu}, it suffices to show that for every $\lambda > 0$ (for some $\lambda > 0$) 
\[
\sup_{\xi \in \R^d}|\widehat{f}(\xi)|e^{M(|\xi|/\lambda)} < \infty. 
\]
Let  $(f_\varepsilon)_\varepsilon$  be a compactly supported representative of $\iota(f)$. Choose $K \Subset \Omega$ such that  $\operatorname{supp} f \subseteq K$ and  $\operatorname{supp} f_\varepsilon \subseteq K$ for all $\varepsilon \in (0,1]$. Let $\psi \in \Df(\Omega)$ such that $\psi \equiv 1$ on a neighborhood of $K$. Hence,
\[ \widehat{f}(\xi) = \langle f(x) - f_\varepsilon(x), \psi(x)e^{i\xi \cdot x}\rangle + \widehat{f}_\varepsilon(\xi), \qquad \forall \xi \in \R^d, \forall \varepsilon \in (0,1].  \]
Let $A ,H$ be the constants appearing in \eqref{assM.2}. Corollary \ref{coro} and the fact that
\[ \|e^{i\xi \cdot x} \|_{\Eh(K)} \leq e^{M(|\xi|/h)}, \qquad \forall \xi \in \R^d, \forall h > 0 \]
imply that 
\[(\exists h > 0)(\exists C > 0)( \exists \varepsilon_0 > 0)( \forall \varepsilon < \varepsilon_0)\]
\[| \langle f(x) - f_\varepsilon(x), \psi(x)e^{i\xi \cdot x}\rangle | \leq C\| \psi\|_{\Eh(K)}e^{M(|\xi|/h)-M(c/h\varepsilon)},\qquad \forall \xi \in \R^d, \]
where $c = 1/(2\sqrt{d}H^2)$. Combining this with Lemma \ref{Pwsalg} yields 
\[ (\exists k,h > 0)(\forall l > 0)(\exists C' > 0)( \exists \varepsilon_0 > 0)( \forall \varepsilon < \varepsilon_0)  \]
\[ ((\exists h,l > 0)(\forall k > 0)(\exists C' > 0)(\exists \varepsilon_0 > 0)( \forall \varepsilon < \varepsilon_0)) \]
\begin{equation}
|\widehat{f}(\xi)| \leq C' (e^{M(|\xi|/h)-M(c/(h\varepsilon))} + e^{M(k/\varepsilon)-M(|\xi|/l)}), \qquad \forall \xi \in \R^d.
\label{mainineq}
\end{equation}
\emph{Beurling case:} Let $\lambda > 0$ be arbitrary. Set $l = \min(\lambda/H, c\min(1,\lambda/h)/(kH^2))$ in \eqref{mainineq}. Hence, for all $|\xi| > c\min(1,\lambda/h)/(\varepsilon_0H)$,  inequality \eqref{mainineq} with $\varepsilon = c\min(1,\lambda/h)/(|\xi|H)$ implies that 
\[|\widehat{f}(\xi)| \leq 2C' Ae^{-M(|\xi|/\lambda)}.\]
\emph{Roumieu case:} Define $\lambda = \max(h,lH)$ and set $k =c/ (\lambda H)$  in \eqref{mainineq}. Hence, for all $|\xi| > c/(\varepsilon_0H)$,  inequality \eqref{mainineq} with $\varepsilon = c/(|\xi|H) $ implies that 
\[|\widehat{f}(\xi)| \leq 2C' Ae^{-M(|\xi|/\lambda)}.\]
\end{proof}

We have the following corollary. As usual, for $f\in{\mathcal{D}^{\ast}}'(\Omega)$, we employ the notation $\operatorname*{sing\: supp}_{\ast} f$ for its classical singular support with respect to $\ast$-ultradifferentiability.

\begin{corollary}
\label{cor sing supp} If $f\in {\mathcal{D}^{\ast}}'(\Omega)$, then $\operatorname*{sing\: supp}_{g,\ast} \iota(f)=\operatorname*{sing\: supp}_{\ast} f.$
\end{corollary}

\section{Wave front sets in $\Gss(\Omega)$}
\label{section wavefronts}
We now study microlocal properties of $\ast$-generalized functions. We introduce a notion of $\ast$-microlocal regularity which is inspired by our definition of the algebra ${\Gss}^{,\infty}$ of regular $\ast$-generalized functions. We shall show at the end of this section that such a notion is a compatible extension of the $\ast$-wave front set \cite{Hormander,Komatsumla,Pilipovic} of an $\ast$-ultradistribution. 

Let $f \in \Gss_c(\Omega)$. The set $\Sigma_{g}^{\ast} (f)\subseteq \R^d \backslash \{0\}$  is defined as the complement in $\R^d \backslash \{0\}$ of the set of all points $\xi_0$ for which there are an open conic neighborhood $\Gamma$ and a compactly supported representative $(f_\varepsilon)_\varepsilon$ of $f$ such that
\[( \exists k > 0)(\forall h > 0) \qquad (( \exists h > 0)(\forall k > 0))    \]
\[ \sup_{\xi \in \Gamma}|\widehat{f}_\varepsilon(\xi)|e^{M(|\xi|/h)} = O(e^{M(k/\varepsilon)}). \]
Note that if the inequality is established for some compactly supported representative, it holds for all such representatives. Furthermore,  $\Sigma_{g}^{\ast}(f)$ is a closed cone and, by Lemma \ref{Pwsalg} and a compactness argument,  $\Sigma_g^{\ast}(f) = \emptyset$ if and only if $f \in \Gsi(\Omega)$.
\begin{lemma}\label{coneshrink}
Let $f \in \Gss_c(\Omega)$ and $u \in \Gsi(\Omega)$. Then, $\Sigma_g^*(uf) \subseteq \Sigma_g^*(f)$. 
\end{lemma}
\begin{proof}
 Let $(f_\varepsilon)_\varepsilon$ be a compactly supported representative of $f$. We have for some $k_0 > 0$ (for every $k_0 > 0$) 
\[ \| \widehat{f}_\varepsilon \|_{L^1}= O(e^{M(k_0/\varepsilon)}). \]
Choose $\psi \in \Ds(\Omega)$ such that $\psi \equiv 1$ in a neighborhood of $\operatorname{supp} f$, hence $uf = u\iota(\psi)f$. Note that $u\iota(\psi) \in \Gsi(\Omega) \cap \Gss_c(\Omega)$ and thus, by Lemma \ref{Pwsalg},
\[( \exists k_1 > 0)(\forall h_1 > 0) \qquad (( \exists h_1 > 0)(\forall k_1 > 0))    \]
\[ \sup_{\xi \in \R^d}|\widehat{v}_\varepsilon(\xi)|e^{M(|\xi|/h_1)} = O(e^{M(k_1/\varepsilon)}), \]
where $(v_\varepsilon)_\varepsilon$ is a compactly supported representative of $u\iota(\psi)$.
Suppose that $\xi_0 \notin \Sigma_g^*(f)$. 
Then, there is an open conic neighborhood $\Gamma$ such that
\[( \exists k_2 > 0)(\forall h_2 > 0) \qquad (( \exists h_2 > 0)(\forall k_2 > 0))    \]
\[ \sup_{\xi \in \Gamma}|\widehat{f}_\varepsilon(\xi)|e^{M(|\xi|/h_2)} = O(e^{M(k_2/\varepsilon)}). \]
Choose an open conic neighborhood $\Gamma_1$ of $\xi_0$ such that $\overline{\Gamma}_1 \subseteq \Gamma \cup \{ 0 \}$. Let  $0 < c < 1$ be smaller than the distance between $\partial \Gamma$ and the intersection of $\Gamma_1$ with the unit sphere. Note that $\{\eta \in \R^d :\, (\exists \xi \in \Gamma_1)(|\xi -\eta|\leq c|\xi|) \} \subseteq \Gamma$ and for all $\xi, \eta \in \R^d$,  $|\xi -\eta|\leq c|\xi|$ implies $|\eta| \geq (1-c)|\xi|$. Hence,
\[ (\exists k_0,k_1,k_2 > 0)(\forall h_1, h_2 > 0)(\exists C, C', C'' > 0)(\exists \varepsilon_0 > 0)(\forall  \varepsilon < \varepsilon_0)(\forall \xi \in \Gamma_1)  \]
\[ ((\exists h_1, h_2 > 0)(\forall k_0,k_1,k_2 > 0)(\exists C, C',C'' > 0)(\exists \varepsilon_0 > 0)(\forall  \varepsilon < \varepsilon_0)(\forall \xi \in \Gamma_1) )\]
\begin{align*}
|\Ft(v_\varepsilon f_\varepsilon)(\xi)| &\leq \frac{1}{(2\pi)^d}\left(\int_{|\eta| \leq c|\xi|} + \int_{|\eta| > c|\xi|} \right) |\widehat{v}_\varepsilon(\eta)| |\widehat{f}_\varepsilon(\xi - \eta)|  \deta \\ 
& \leq  \frac{\| \widehat{v}_\varepsilon \|_{L^1}}{(2\pi)^d} \sup_{|\xi-\eta| \leq c|\xi|}|\widehat{f}_\varepsilon(\eta)| + C \| \widehat{f}_\varepsilon \|_{L^1} e^{-M(c|\xi|/h_1) + M(k_1/\varepsilon)}  \\
& \leq  C'e^{-M((1-c)|\xi|/h_2) + M(k_2/\varepsilon) + M(k_1/\varepsilon)} + C'' e^{-M(c|\xi|/h_1) + M(k_1/\varepsilon) +M(k_0/\varepsilon)}.
\end{align*}
\end{proof}
Let $f \in \Gss(\Omega)$. The generalized wave front set $WF_{g,\ast}(f)$ is defined as the complement in $\Omega \times ( \R^d \backslash \{ 0 \})$ of all pairs $(x_0, \xi_0) $ for which there is $\psi \in \Ds(\Omega)$, with $\psi \equiv 1$ in a neighborhood of $x_0$, such that $\xi_0 \notin \Sigma_g^*(\iota(\psi)f)$. Note that in the Roumieu case, one may always take $\psi \in \Df(\Omega)$, as easily follows from Lemma \ref{coneshrink}. As in \cite[Sect. 8.1]{Hormander}, Lemma \ref{coneshrink} also implies that the projection of $WF_{g,\ast}(f)$ on $\Omega$ is $\operatorname{sing\: supp}_{g,\ast}(f)$, while its projection on $\R^d \backslash \{ 0 \}$ is $\Sigma_g^*(f)$ if $f\in\Gss_{c}(\Omega)$. We collect some properties of $WF_{g,\ast}$ in the next proposition.

\begin{proposition}
Let $f \in \Gss(\Omega)$. Then,
\begin{enumerate}
\item[$(i)$] $WF_{g,\ast}(uf) \subseteq WF_{g,\ast}(f)$ for all  $u \in \Gsi(\Omega)$.
\item[$(ii)$] $WF_{g,\ast}(P(D)f) \subseteq WF_{g,\ast}(f)$ for any $\ast$-ultradifferential operator $P(D)$.
\end{enumerate}
\end{proposition}
\begin{proof}
The first assertion follows from Lemma \ref{coneshrink}.  Let us prove $(ii)$. One readily shows that
\begin{equation}
\Sigma_g^*(P(D)u) \subseteq \Sigma_g^*(u), \qquad \forall u \in \Gss_c(\Omega).  
\label{ultraoperator}
\end{equation}  
Now suppose $(x_0, \xi_0) \notin WF_{g,\ast}(f)$. Then, there is $\psi \in \Ds(\Omega)$, with $\psi \equiv 1$ in a neighborhood of $x_0$, such that $\xi_0 \notin \Sigma_g^*(\iota(\psi)f)$. Choose $\chi \in \Ds(\Omega)$ such that $\chi \equiv 1$ in a neighborhood of $x_0$ and $\psi \equiv 1$ in a neighborhood of $\operatorname{supp} \chi$. Since $P(D): \Gss \rightarrow \Gss$ is a sheaf morphism, we have $\iota(\chi)P(D)f = \iota(\chi)P(D)(\iota(\psi) f)$ and thus, by Lemma \ref{coneshrink} and inclusion \eqref{ultraoperator},
\[ \xi_0 \notin \Sigma_g^*(\iota(\psi)f) \supseteq \Sigma_g^*(P(D)(\iota(\psi)f))  \supseteq \Sigma_g^*(\iota(\chi)P(D)(\iota(\psi)f)) =   \Sigma_g^*(\iota(\chi)P(D)f). \]   
\end{proof}
We now compare $WF_{g,\ast}(f)$ with its classical counterpart for ultradistributions. We follow the standard definition \cite{Komatsumla, Pilipovic} for the wave front set $WF_{\ast}(f)$ of $f\in{\Ds}'(\Omega)$. Recall that  $WF_{\ast}(f)$  is the complement in $\Omega \times ( \R^d  \backslash \{ 0 \})$ of the set of all $(x_0, \xi_0) $ for which there are an open conic neighborhood $\Gamma$ of $\xi_0$ and $\psi \in \Ds(\Omega)$, with $\psi \equiv 1$ in a neighborhood of $x_0$, such that for every $\lambda > 0$ (for some $ \lambda >0$)
\[  \sup_{\xi \in \Gamma} |\widehat{\psi f}(\xi)| e^{M(|\xi|/\lambda)} < \infty.  \]
We have the the following equality:
\begin{theorem}
Let $f \in {\Ds}'(\Omega)$. Then,  $WF_*(f) = WF_{g,\ast}(\iota(f))$. 
\end{theorem}
\begin{proof}
Let $(x_0, \xi_0) \notin WF_*(f)$. Find an open conic neighborhood $\Gamma$ of $\xi_0$ and $\chi \in \Ds(\Omega)$, with $\chi \equiv 1$ in a neighborhood of $x_0$, such that for every $\lambda > 0$ (for some $ \lambda >0$)
\[  \sup_{\xi \in \Gamma} |\widehat{\chi f}(\xi)| e^{M(|\xi|/\lambda)} < \infty.  \]
Choose $\psi \in \Ds(\Omega)$ such that $\psi \equiv 1$ in a neighborhood of $x_0$ and $\chi \equiv 1$ in a neighborhood of $\operatorname{supp} \psi$. We show that $\xi_0 \notin \Sigma_g^*(\iota(\psi)\iota(f))$. Theorem \ref{main} gives $\iota(\psi)\iota(f) = \iota(\psi)\iota(\chi f)= [(\psi(\chi f \ast \phi_\varepsilon))_\varepsilon]$. By \cite[Lemm. 3.3]{Komatsu} we have that for every $h > 0$ (for some $h > 0$)
\[  \sup_{\xi \in \R^d} |\widehat{\psi}(\xi)| e^{M(|\xi|/h)} < \infty,\]
and, by \cite[Thrm. 2.4]{Cho}, that for some $k > 0$ (for every $k > 0$)
\[  \sup_{\xi \in \R^d} |\widehat{\chi f}(\xi)| e^{-M(k|\xi|)} < \infty.\]
Select an open conic neighborhood $\Gamma_1$ of $\xi_0$ exactly as in the proof of Lemma \ref{coneshrink}. We then have
\[ (\exists k > 0)(\forall h,\lambda > 0)(\exists C, C' > 0)( \forall \varepsilon \leq 1)(\forall \xi \in \Gamma_1)  \]
\[ ((\exists h,\lambda > 0)(\forall k > 0)(\exists C, C' > 0)( \forall  \varepsilon \leq 1)(\forall \xi \in \Gamma_1))\]
\begin{align*}
&|\Ft(\psi(\chi f \ast \phi_\varepsilon))(\xi)| \\ 
&\leq \frac{1}{(2\pi)^d}\left(\int_{|\eta| \leq c|\xi|} + \int_{|\eta| > c|\xi|} \right) |\widehat{\psi}(\eta)| |\widehat{\chi f}(\xi - \eta)| |\widehat{\phi}(\varepsilon(\xi - \eta))| \deta \\ \\
& \leq  \frac{\| \widehat{\phi}\|_{L^\infty}\| \widehat{\psi}\|_{L^1}}{(2\pi)^d} \sup_{|\xi-\eta| \leq c|\xi|}|\widehat{\chi f}(\eta)| + \frac{C}{(2\pi)^d} e^{-M(c|\xi|/h)}\int_{\R^d}e^{M(k|t|)} |\widehat{\phi}(\varepsilon t)| \dt \\ \\
& \leq  \frac{C'\| \widehat{\phi}\|_{L^\infty}\| \widehat{\psi}\|_{L^1}}{(2\pi)^d}e^{-M((1-c)|\xi|/\lambda)} + \frac{C\|\widehat{\phi}\|_{L^{1}}}{(2\pi \varepsilon)^d} e^{-M(c|\xi|/h) + M(2k/\varepsilon)},
\end{align*}
which shows that $(x_0,\xi_o)\notin WF_{g,\ast}(\iota(f))$. Conversely, let $(x_0, \xi_0) \notin WF_{g,\ast}(\iota(f))$. Find $\psi \in \Df(\Omega)$, with $\psi \equiv 1$ in a neighborhood of $x_0$,  such that $\xi_0 \notin \Sigma_g^*(\iota(\psi)\iota(f))$.
Let $\chi \in \Ds(\Omega)$ such that $\chi \equiv 1$ in a neighborhood of $\operatorname{supp} \psi$. Since we have a sheaf monomorphism ${\Ds}' \rightarrow \Gss$, we conclude that $\iota(\psi)\iota(f) = \iota(\psi)\iota(\chi f)= [(\psi u_\varepsilon)_\varepsilon]$, where $(u_\varepsilon)_\varepsilon$ is a compactly supported representative of $\iota(\chi f)$. Hence, there is an open conic neighborhood $\Gamma$  of $\xi_0$ such that
\[( \exists k > 0)(\forall h > 0) \qquad  (( \exists h > 0)(\forall k > 0))    \]
\begin{equation}
\sup_{\xi \in \Gamma}|\widehat{\psi u_\varepsilon}(\xi)|e^{M(|\xi|/h)} = O(e^{M(k/\varepsilon)}). 
\label{boundwf}
\end{equation}
Note that
\[ \widehat{\psi f}(\xi) =  \widehat{\psi \chi f}(\xi) =  \ \langle \chi f(x) - u_\varepsilon(x), \psi(x)e^{i\xi \cdot x}\rangle + \widehat{\psi u }_\varepsilon(\xi), \qquad \forall \xi \in \R^d, \forall \varepsilon \in (0,1].  \]
Similarly as in the proof of Theorem \ref{mainreg}, but using relation \eqref{boundwf} instead of Lemma \ref{Pwsalg}, one can show that for every $\lambda > 0$ (resp. for some $ \lambda >0$)
\[  \sup_{\xi \in \Gamma} |\widehat{\psi f}(\xi)| e^{M(|\xi|/\lambda)} < \infty;  \]
we leave details to the reader. This concludes the proof of the theorem.
\end{proof}

\section{Embeddings of Beurling-Bj\"orck ultradistributions into algebras}\label{section weight functions}
In this last section we outline how our ideas from the previous Sections \ref{section-embedding}--\ref{section wavefronts} can be adapted to develop an analogous nonlinear theory for ultradistributions defined via weight functions. We follow here the Beurling-Bj\"{o}rck approach to ultradistribution theory \cite{Bjorck} (see also \cite{b-m-t,Cior}). 

\subsection{Spaces defined via weight functions}  We review in this preparatory subsection some basic properties of the Beurling-Bj\"{o}rck spaces. We always assume that $\omega:[0,\infty)\to[0,\infty)$ is a non-decreasing weight function satisfying $\omega(0)=0$ and the ensuing three conditions:
\begin{itemize}
\item[$(\alpha)$]  $\omega(t_1 + t_2) \leq \omega(t_1) + \omega(t_2), \quad \forall t_1,t_2\geq0,$
\item[$(\beta)$] $\displaystyle \int_{1}^{\infty} \frac{\omega(t)}{t^{2}} \mathrm{d}t < \infty,$
\item[$(\gamma)$]  $\omega(t) \geq b\log(1+t) + a, \quad \forall t\geq 0$, for some $a \in \R$ and $b > 0$. \end{itemize}
The meaning of these three conditions is explained in \cite{Bjorck}. 

Let $\Omega \subseteq \R^d$ be open, $K \Subset\Omega$ and $\lambda>0$. The Banach space $\mathcal{D}_{\omega}^\lambda(K)$ consists of those $\phi  \in L^{1}(\R^d)$ such that $\operatorname{supp} \phi \subseteq K$ and
\begin{equation}\label{eq norm 1}
\| \phi \|_\lambda= \| \phi \|_{\mathcal{F}L^1_\omega,\lambda}:= \int_{\R^d} |\widehat{\phi}(\xi)|e^{\lambda \omega(|\xi|)} d\xi < \infty.
\end{equation}
Set further
\[ \mathcal{D}_{(\omega)}(\Omega) = \varinjlim_{K \Subset \Omega} \varprojlim_{\lambda \to \infty} \mathcal{D}_{\omega}^\lambda(K), \quad \mathcal{D}_{\{\omega\}}(\Omega) = \varinjlim_{K \Subset  \Omega} \varinjlim_{\lambda \to 0} \mathcal{D}_{\omega}^\lambda(K).\]
Condition $(\gamma)$ yields $\mathcal{D}_{(\omega)}(\Omega) \subseteq \mathcal{D}(\Omega)$. In order to ensure that $\mathcal{D}_{\{\omega\}} (\Omega)\subseteq \mathcal{D}(\Omega)$, we impose the following additional condition on $\omega$ in the Roumieu case:
\begin{itemize}
\item[$(\gamma_0)$] $\displaystyle\lim_{t\to\infty} \frac{\omega(t)}{\log (1 + t)}=\infty$.
\end{itemize}
Their duals $\mathcal{D}'_{(\omega)}(\Omega)$ and $\mathcal{D}'_{\{\omega\}}(\Omega)$  are the ultradistribution space of class $(\omega)$ (Beurling type) and class $\{\omega\}$ (Roumieu type), respectively. We write in this  $\ast=(\omega)$ or $\{\omega \}$ in order to treat both cases simultaneously. Note that if $\omega(t)=\log(1+t)$, one then recovers the classical Schwartz spaces as particular instances of the Beurling case.

The space $\mathcal{E}_{\ast}(\Omega)$ of $\ast$-ultradifferentiable functions is defined  as the space of multipliers of $\mathcal{D}_{\ast}(\Omega)$, that is, 
 $$
 \mathcal{E}_{\ast}(\Omega) = \{ f \in \mathcal{D}'_{\ast}(\Omega) \, : \, f\phi \in \mathcal{D}_{\ast}(\Omega) \mbox { for all } \phi \in \mathcal{D}_{\ast}(\Omega)\}.
$$
Naturally,  $\mathcal{E}_{\ast}(\Omega)\subseteq C^{\infty}(\Omega)$, with equality if $\omega(t)=\log(1+t)$ and $\ast=(\omega)$. 
We endow $\mathcal{E}_{\ast}(\Omega)$ with the coarsest topology for which all the mappings $f \to f\phi$, $\phi \in \mathcal{D}_{\ast}(\Omega)$, are continuous. More explicitly,
 $$
 \mathcal{E}_{(\omega)}(\Omega)= \varprojlim_{K \Subset \Omega} \varprojlim_{\lambda \to \infty} \mathcal{E}_{\omega}^\lambda(K) , \quad  \mathcal{E}_{\{\omega\}}(\Omega)= \varprojlim_{K \Subset \Omega} \varinjlim_{\lambda \to 0} \mathcal{E}_{\omega}^\lambda(K),
 $$
 where  
$$
\mathcal{E}_{\omega}^\lambda(K) = \{ f \in C^{\infty}(\Omega) : \, \|f\|_{\mathcal{E}_{\omega}^\lambda(K)}:= \inf\{ \| \phi \|_\lambda \, : \, \phi  \in \mathcal{D}(\Omega) \mbox{ with } f_{|K}= \phi_{|K}\}< \infty \}.
$$
The dual $\mathcal{E}'_{\ast}(\Omega)$ is then the subspace of $\mathcal{D}'_{\ast}(\Omega)$ consisting of compactly supported $\ast$-ultradistributions. 

We end this subsection with a technical remark about the family of norms (\ref{eq norm 1}). This remark will play an important role for our estimates in the next subsections.

\begin{remark} \label{L2estimate}
The topology on $\mathcal{D}_{\ast}(\Omega)$ can be equivalently induced by Fourier-Lebesgue $L^\infty$- or $L^2$-type norms. More precisely, define
$$\|\phi \|_{\mathcal{F}L^\infty_\omega,\lambda}:= \sup_{\xi \in \R^d} |\widehat{\phi}(\xi)|e^{\lambda \omega(|\xi|)},\quad \| \phi \|_{\mathcal{F}L^2_\omega,\lambda}:= \left(\int_{\R^d} |\widehat{\phi}(\xi)|^2 e^{2\lambda \omega(|\xi|)} d\xi \right)^{1/2}.
$$
For all compactly supported $\phi \in L^{1}(\mathbb{R}^{d})$, we have 
\begin{equation}
C_1\|\phi \|_{\mathcal{F}L^\infty_\omega,\lambda} \leq \|\phi \|_{\mathcal{F}L^1_\omega,\lambda} \leq C_2\|\phi \|_{\mathcal{F}L^\infty_\omega,\lambda + \Lambda}, \qquad \lambda > 0,
\label{equivalentnorms}
\end{equation}
for $\Lambda = (d+1)/b$, where $b$ is the constant ocurring in $(\gamma)$, and some $C_1=C_{1,\lambda}>0$ and $C_2>0$; the first inequality is non-trivial and follows by inspection in the proof of \cite[Thrm.\ 1.4.1]{Bjorck}, while the second one follows directly from $(\gamma)$. Similarly, we have 
$$
C'_1\|\phi \|_{\mathcal{F}L^2_\omega,\lambda} \leq \|\phi \|_{\mathcal{F}L^1_\omega,\lambda} \leq C'_2\|\phi \|_{\mathcal{F}L^2_\omega,\lambda + \Lambda/2}, \qquad \lambda > 0,
$$
for some $C'_1=C'_{1,\lambda}>0$ and $C'_2>0$; the first inequality being a consequence of \eqref{equivalentnorms},
 while the second one follows from the Cauchy-Schwartz inequality and $(\gamma)$. Moreover, if $\omega$ satisfies $(\gamma_0)$, the following inequalities hold, which are sharper for small values of $\lambda$,
$$
 C_1\|\phi \|_{\mathcal{F}L^\infty_\omega,\lambda} \leq \|\phi \|_{\mathcal{F}L^1_\omega,\lambda} \leq C_2\|\phi \|_{\mathcal{F}L^\infty_\omega,2\lambda}, \qquad \lambda > 0,
$$
and
$$
C'_1\|\phi \|_{\mathcal{F}L^2_\omega,\lambda} \leq \|\phi \|_{\mathcal{F}L^1_\omega,\lambda} \leq C'_2\|\phi \|_{\mathcal{F}L^2_\omega,2\lambda}, \qquad \lambda > 0,
$$
for some constants depending on $\lambda$. 
\end{remark}

\subsection{Algebras of generalized functions of class $(\omega)$ and $\{\omega\}$} For an open subset $\Omega$ of $\R^d$,  we define the differential algebras (with pointwise multiplication of nets)
\begin{align*}
\mathcal{E}_{(\omega), \mathcal{M}}(\Omega) = \{ (f_\varepsilon)_\varepsilon \in \mathcal{E}_{(\omega)}(\Omega)^{(0,1]} \, : \, &(\forall K \Subset \Omega) (\forall \lambda > 0) (\exists k > 0) \\
 &\|f_\varepsilon\|_{\mathcal{E}^\lambda_{\omega}(K)}  = O(e^{k\omega(1/\varepsilon)})\},
\end{align*}
\begin{align*}
\mathcal{E}_{\{\omega\}, \mathcal{M}}(\Omega) = \{ (f_\varepsilon)_\varepsilon \in \mathcal{E}_{\{\omega\}}(\Omega)^{(0,1]} \, : \, &(\forall K \Subset \Omega) (\forall k > 0) (\exists \lambda > 0) \\
 &\|f_\varepsilon\|_{\mathcal{E}^\lambda_{\omega}(K)}  = O(e^{k\omega(1/\varepsilon)})\},
\end{align*}
and their ideals
\begin{align*}
\mathcal{E}_{(\omega), \mathcal{N}}(\Omega) =\{ (f_\varepsilon)_\varepsilon \in \mathcal{E}_{(\omega)}(\Omega)^{(0,1]} \, : \, &(\forall K \Subset \Omega) (\forall k > 0) (\forall \lambda > 0) \\
 &\|f_\varepsilon\|_{\mathcal{E}^\lambda_{\omega}(K)}  = O(e^{-k\omega(1/\varepsilon)})\},
 \end{align*}
 \begin{align*}
\mathcal{E}_{\{\omega\}, \mathcal{N}}(\Omega) = \{ (f_\varepsilon)_\varepsilon \in \mathcal{E}_{\{\omega\}}(\Omega)^{(0,1]} \, : \, &(\forall K \Subset \Omega) (\exists k > 0) (\exists \lambda > 0) \\
 &\|f_\varepsilon\|_{\mathcal{E}^\lambda_{\omega}(K)}  = O(e^{-k\omega(1/\varepsilon)})\}.
\end{align*}

The corresponding algebra  $\mathcal{G}_\ast(\Omega)$ of \emph{generalized functions of class} $\ast$ ($\ast$-generalized functions) is  the factor algebra
\[ \mathcal{G}_\ast(\Omega) = \mathcal{E}_{\ast, \mathcal{M}}(\Omega)/\mathcal{E}_{\ast, \mathcal{N}}(\Omega). \] 
Since the partial derivatives are continuous operators on $\mathcal{E}_{\ast}(\Omega)$, $\mathcal{G}_\ast(\Omega)$  is a differential algebra, with the action of differential operators being defined as $
\partial^{\alpha} [(f_\varepsilon)_\varepsilon] = [(\partial^{\alpha}f_\varepsilon)_\varepsilon],$ $\alpha\in\mathbb{N}^{d}$. 
Note that $\mathcal{E}_{\ast}(\Omega)$ is a subalgebra of $\mathcal{G}_\ast(\Omega)$ via the constant canonical embedding
\begin{equation}\label{eq constant embedding 2}
 \sigma(f) = [(f)_\varepsilon], \qquad \forall f \in \mathcal{E}_{\ast}(\Omega). 
 \end{equation}
As in the case of weight sequences discussed in Section \ref{section algebra}, the functor $\Omega \rightarrow \mathcal{G}_\ast(\Omega)$ is a fine and supple but non-flabby sheaf of differential algebras on any open subset of $\R^d$.

It is worth noticing that if $\omega(t)=\log (1+t)$, the algebra
$\mathcal{G}_{(\omega)}(\Omega)$ coincides with the Colombeau special
algebra \cite{GGK}. Since we will embed $\mathcal{D}_{\ast}(\Omega)$ into
$\mathcal{G}_{\ast}(\Omega)$ in the next subsection, our considerations with
weight functions therefore provide a unified nonlinear approach to treat
distributions and ultradistributions simultaneously.

 We also have a null characterization of the negligible nets. 

 \begin{proposition} \label{nullcharweights}
 Let $(f_\varepsilon)_\varepsilon \in \mathcal{E}_{\ast, \mathcal{M}}(\Omega)$. Then, $(f_\varepsilon)_\varepsilon \in \mathcal{E}_{\ast, \mathcal{N}}(\Omega)$ if and only if 
\[(\forall K \Subset \Omega) ( \forall k > 0)\qquad ((\forall K \Subset \Omega) ( \exists k > 0))    \]
\[ \sup_{x \in K}|f_\varepsilon(x)| = O(e^{-k\omega(1/\varepsilon)}). \]
 \end{proposition}
\begin{proof}  This follows from the following observations, two Landau type inequalities. Let $K \Subset K' \Subset \Omega$. Then for each $\lambda > 0$ there is $C>0$ 
$$
\|f\|_{\mathcal{E}^\lambda_\omega(K)} \leq C\sup_{x \in K'}|f(x)|^{1/2} \|f\|^{1/2}_{\mathcal{E}^{2\lambda + \Lambda}_\omega(K')}, \quad f \in \mathcal{E}_\ast(\Omega),
$$ 
where $\Lambda$ is the constant from Remark \ref{L2estimate}. If in addition $\omega$ satisfies $(\gamma_0)$, then for each $\lambda > 0$ there is $C>0$ such that
$$
\|f\|_{\mathcal{E}^\lambda_\omega(K)} \leq C\sup_{x \in K'}|f(x)|^{1/2} \|f\|^{1/2}_{\mathcal{E}^{4\lambda}_\omega(K')}.
$$ 
We  show the first inequality, the second one can be proved in a similar fashion. Let $\psi \in \mathcal{D}_{(\omega)}(\Omega)$ with $\psi \equiv 1$ in a neighbourhood of $K$ and $\operatorname{supp} \psi \subseteq K'$. Let $\phi \in \mathcal{D}_{(\omega)}(\Omega)$ with $\phi_{|K'} = f_{|K'}$ be arbitrary. Then Remark \ref{L2estimate} implies that\begin{align*}
\|f\|_{\mathcal{E}^\lambda_\omega(K)} &\leq C\| \phi \psi\|_{\mathcal{F}L^2_\omega, \lambda + \Lambda/2} \\
&\leq C\| \widehat{\phi\psi}\|^{1/2}_{L^\infty} \| \phi \psi\|^{1/2}_{\mathcal{F}L^1_\omega, 2\lambda + \Lambda} \\
&\leq C\|\psi\|^{1/2}_{L^1} \|\psi\|^{1/2}_{\mathcal{F}L^1_\omega, 2\lambda + \Lambda} \sup_{x \in K'}|f(x)|^{1/2}  \| \phi\|^{1/2}_{\mathcal{F}L^1_\omega, 2\lambda + \Lambda},
\end{align*}
for some $C=C_{\lambda}>0$.

\end{proof}
Proposition \ref{nullcharweights} enables one to obtain a pointwise characterization of $\ast$-generalized functions, that is, a version of Proposition \ref{pointwise} for the algebra $\mathcal{G}_{\ast}(\Omega)$. We leave the precise formulation of this result and the necessary definitions to the reader.
\subsection{Embedding of $\omega$-ultradistributions} Our strategy to embed $\mathcal{D}'_{\ast}(\Omega)$ into the differential algebra $\mathcal{G}_{\ast}(\Omega)$ is the same as that from Section \ref{section-embedding}, namely, we will first embed $\mathcal{E}_{\ast}(\Omega)$ into $\mathcal{G}_{\ast,c}(\Omega)$, where the latter denotes the ideal of compactly supported sections of $\mathcal{G}_{\ast}$ on $\Omega$. For it, we shall also employ a $(N_p)$-net of mollifiers (cf. Section \ref{section-embedding}). 

This time our assumptions on the weight sequence $N_p$ are $(M.1)$, $(M.3)'$ and the ensuing condition: There is another weight sequence $M_p$ satisfying $(M.1)$, the condition (\ref{eqass3N}), and 
\begin{equation}
\label{eqass4N}
\lim_{t\to\infty}\frac{\omega(t)}{M(kt)}=0, \quad \mbox{for each } k>0.
\end{equation}
In particular, note that we have the continuous and dense embeddings 
$$\mathcal{D}^{(N_p)}(\Omega) \hookrightarrow\mathcal{D}^{(M_p)}(\Omega) \hookrightarrow \mathcal{D}_{\ast}(\Omega).$$
That such a choice of $N_p$ is always possible follows by combining results on majorants by Beurling (cf. remark after \cite[Thrm.\ 1.2.7]{Bjorck}), Cior\v{a}nescu and Zsid\'o \cite[Thrm.\ 1.8]{Cior}, and Roumieu (Lemma \ref{Roum} above).

\begin{proposition}\label{embeddingweight} 
The mapping
\[ \iota_c: \mathcal{E}'_\ast(\Omega) \rightarrow \mathcal{G}_{\ast,c}(\Omega):\quad  f \rightarrow \iota_c(f) = [((f \ast \phi_\varepsilon)_{|\Omega})_\varepsilon] \]
is a support preserving linear embedding, that is,  $\operatorname{supp}\iota_{c}(f) = \operatorname{supp} f$ for all $f \in \mathcal{E}'_\ast(\Omega)$. Moreover,
$\iota_{c|\mathcal{D}_\ast(\Omega)} = \sigma$, where $\sigma$ is the constant embedding $(\ref{eq constant embedding 2})$.
\end{proposition}
\begin{proof}
Let $f \in \mathcal{E}'_\ast(\Omega)$. We first show that $((f \ast \phi_\varepsilon)_{|\Omega})_\varepsilon \in \mathcal{E}_{\ast, \mathcal{M}}(\Omega)$.  Let $K \Subset  \Omega$ and $\lambda > 0$ be arbitrary. Let $\chi \in \mathcal{D}_{(\omega)}(\Omega)$ such that $\chi \equiv 1$ on a neighborhood of $K$. The Paley-Wiener theorem for Beurling-Bj\"orck ultradistributions\footnote{It should be noticed that Bj\"orck only considers the Beurling case; however, his proof can be easily adapted to treat the Roumieu case as well.} \cite[Thrm. 1.8.14]{Bjorck}  implies that there is $k > 0$ (for each  $k > 0$)   $|\widehat{f}(\xi)|\leq Ce^{k\omega(|\xi|)}$ for some $C = C_k > 0$. We have,
\begin{align*}
\| f \ast \phi_\varepsilon\|_{\mathcal{E}^\lambda_\omega(K)} &\leq 
\|\chi (f\ast \phi_{\varepsilon}) \|_{\mathcal{F}L^1_\omega,\lambda}\\
&\leq  \frac{C\|\chi\|_{\mathcal{F}L^1_\omega,\lambda}}{(2\pi)^{d}} \int_{\mathbb{R}^{d}} |\widehat{\phi}(\varepsilon \xi)| e^{(\lambda+k)\omega(|\xi|)}\mathrm{d}\xi\\
&\leq \frac{C\|\chi\|_{\mathcal{F}L^1_\omega,\lambda}\|\widehat{\phi}\|_{L^1}}{(2\pi\varepsilon)^{d}}\:e^{2(\lambda+k) \omega(1/\varepsilon)}, \\
\end{align*}
which in view of condition ($\gamma$) (condition ($\gamma_0$)) shows the assertion. The injectivity of $\iota_c$ is clear. This already yields $\operatorname*{supp}f \subseteq \operatorname*{supp}\iota_{c}( f)$ and thus the mapping has range in $\mathcal{G}_{\ast,c}(\Omega)$. The reverse inclusion $\operatorname*{supp}\iota_{c}(f) \subseteq \operatorname*{supp}f$ actually follows from the continuous embedding $\mathcal{E}'_\ast(\Omega)\to {\mathcal{E}^{(M_p)}}'(\Omega)$, Proposition \ref{support}, the assumption \ref{eqass4N}, and Proposition \ref{nullcharweights}. Finally, let $f\in\mathcal{D}_{\ast}(\Omega)$. For each  $\lambda>0$ (some $\lambda>0$), we have
\begin{align*}
\|f \ast \phi_\varepsilon - f\|_{L^\infty} & \leq \frac{1}{(2\pi)^d}\int_{\R^d} |\widehat{f}(\xi)||\widehat{\phi}(\varepsilon \xi) - 1| \dxi \\
& \leq \frac{\|f\|_{\mathcal{F}L^\infty_\omega,\lambda +\lambda'}}{(2\pi)^d}(1+ \|\widehat{\phi}\|_{L^\infty}) e^{-\lambda\omega(1/\varepsilon)} \int_{\R^d} e^{-\lambda'\omega(|\xi|)}  \dxi, \\
\end{align*}
where in the Beurling case $\lambda'=\Lambda$ (cf. Remark \ref{L2estimate}) and in the Roumieu case $\lambda'=\lambda$. Proposition \ref{nullcharweights} and condition $(\gamma)$ (condition $(\gamma_0)$) imply that $\iota_{c}(f)=\sigma(f)$. 
\end{proof}

As in Section \ref{section-embedding}, Proposition \ref{embeddingweight} and the fact that $\mathcal{D}'_{\ast}$ and $\mathcal{G}_{\ast}$ are fine sheaves allows one to automatically extend the embedding $\iota_{c}:\mathcal{E}'_{\ast}(\Omega)\to \mathcal{G}_{\ast,c}(\Omega)$ to a unique sheaf monomorphism:
\begin{theorem}\label{mainweights}
There exists a linear embedding  $\iota = \iota_\Omega: \mathcal{D}'_\ast(\Omega) \rightarrow \mathcal{G}_\ast(\Omega)$ satisfying:
\begin{enumerate}
\item[$(i)$] $\iota_{|\mathcal{E}'_\ast(\Omega) } = \iota_c.$
\item[$(ii)$] $\iota$ commutes with $\partial^{\alpha}$ for each $\alpha\in\mathbb{N}^{d}$, namely,
\[ \partial^{\alpha}\iota(f) = \iota(\partial^{\alpha} f), \quad f \in \mathcal{D}'_\ast(\Omega).\]
\item[$(iii)$] $\iota_{|\mathcal{E}'_\ast(\Omega)}$ coincides with the constant embedding. Consequently, 
$$
\iota(fg) = \iota(f)\iota(g), \quad f,g \in \mathcal{E}_\ast(\Omega).
$$ 
\end{enumerate}
Furthermore, the entirety of all $\iota_\Omega: \mathcal{D}'_\ast(\Omega) \rightarrow \mathcal{G}_\ast(\Omega)$ is a sheaf monomorphism $\mathcal{D}'_\ast \rightarrow \mathcal{G}_\ast$ on any open subset of $\mathbb{R}^{d}$.
\end{theorem}

We point out that the technique employed at the end of the proof of Proposition \ref{embeddingweight} also yields the following version of Corollary \ref{coro}.
\begin{lemma}\label{coroweights} 
Let $f \in \mathcal{E}'_{\ast}(\Omega)$, $(f_\varepsilon)_\varepsilon$ a representative of $\iota_c(f)$ and $K \Subset \Omega$. Then
\[(\exists \lambda > 0)(\exists C > 0)( \exists \varepsilon_0 > 0)(\forall \psi \in \mathcal{D}_{(\omega)}(K))( \forall \varepsilon < \varepsilon_0)\]
\[| \langle f -f_\varepsilon, \psi \rangle | \leq C\| \psi\|_{\mathcal{F}L^\infty_\omega, 2\lambda}\: e^{-\lambda\omega(1/\varepsilon)}.\]
\end{lemma}
Finally, the next remark discusses the optimality of our embedding in some important cases.
\begin{remark}
If the function $\omega(e^t)$ is convex -- a crucial assumption in the Braun-Meise-Taylor approach to $\omega$-ultradistributions \cite{b-m-t}-- one can define the notion of an ultradifferentiable operator of class $(\omega)$ or $\{\omega\}$  and show an analogue of Komatsu's second structure theorem (as stated in \cite{Taki}) for $\omega$-ultradistributions \cite{Braun}. In such a case, it is therefore possible to state and prove a Schwartz' impossibility type result for $\omega$-ultradistributions (cf. Section \ref{imposs section}). Moreover, every ultradifferentiable operator of class $\ast$ canonically defines a linear operator on $\mathcal{G}_\ast(\Omega)$ and one can show that the embedding constructed below is again optimal in view of the impossibility result.
\end{remark} 

\subsection{Microlocal analysis in $\mathcal{G}_\ast(\Omega)$} The corresponding subalgebra  $\mathcal{G}^\infty_\ast(\Omega)$ of \emph{regular} $\ast$-generalized functions is the factor algebra
\[ \mathcal{G}^\infty_\ast(\Omega) =\mathcal{E}^\infty_{\ast,\mathcal{M}}(\Omega)/\mathcal{E}_{\ast,\mathcal{N}} (\Omega) ,\]
where
\begin{align*}
\mathcal{E}^\infty_{(\omega), \mathcal{M}}(\Omega) = \{ (f_\varepsilon)_\varepsilon \in \mathcal{E}_{(\omega)}(\Omega)^{(0,1]} \, : \, &(\forall K \Subset \Omega) (\exists k > 0)(\forall \lambda > 0)  \\
 &\|f_\varepsilon\|_{\mathcal{E}^\lambda_{\omega}(K)}  = O(e^{k\omega(1/\varepsilon)}) \},
\end{align*}
and
\begin{align*}
\mathcal{E}^\infty_{\{\omega\}, \mathcal{M}}(\Omega) = \{ (f_\varepsilon)_\varepsilon \in \mathcal{E}_{\{\omega\}}(\Omega)^{(0,1]} \, : \, &(\forall K \Subset \Omega) (\exists \lambda > 0) (\forall k > 0) \\
 &\|f_\varepsilon\|_{\mathcal{E}^\lambda_{\omega}(K)}  = O(e^{k\omega(1/\varepsilon)}) \}.
\end{align*}
Clearly, $\mathcal{G}_{\ast}^{\infty}$ is a fine and supple (but non-flabby) subsheaf of $\mathcal{G}_{\ast}$.

In view of Lemma \ref{coroweights} and the Paley-Wiener-Bj\"orck theorem \cite{Bjorck}, our argument used in the proof of Theorem \ref{mainreg} can be easily modified to obtain the \emph{regularity theorem}
\begin{equation}
\label{eq regularity 2}
\mathcal{G}^\infty_\ast(\Omega) \cap \iota(\mathcal{D}'_\ast(\Omega)) = \iota(\mathcal{E}_\ast(\Omega)). 
\end{equation}
The definitions of the $\ast$-singular support and the generalized one, denoted as $\operatorname*{sing\: supp}_{\ast} $ and $\operatorname*{sing \:supp}_{g,\ast} $, with respect $\mathcal{E}_{\ast}$- and $\mathcal{G}_{\ast}^{\infty}$-regularity should be clear. As a corollary of (\ref{eq regularity 2}), one obtains
$$
\operatorname{sing\: supp}_{g,\ast}\iota(f)= \operatorname{sing \:supp}_{\ast} f, \quad \forall f\in{\mathcal{D}}'_{\ast}(\Omega).
$$

Based upon the notion of $\mathcal{G}_\ast^\infty$-regularity, we now define the wave front set of a $\ast$-generalized function. We begin by recalling the case of ultradistributions. For $f \in \mathcal{D}'_\ast(\Omega)$, the wave front set  $WF_\ast(f)$  is defined \cite{d-v,f-g-j} as the complement in $\Omega \times ( \R^d  \backslash \{ 0 \})$ of the set all $(x_0, \xi_0) $ for which there are an open conic neighborhood $\Gamma$ of $\xi_0$ and $\psi \in \mathcal{D}_\ast(\Omega)$, with $\psi \equiv 1$ in a neighborhood of $x_0$, such that for every $\lambda > 0$ (for some $ \lambda >0$)
\[  \sup_{\xi \in \Gamma} |\widehat{\psi f}(\xi)| e^{\lambda\omega(|\xi|)} < \infty.  \]
Note that, equivalently, we may simply ask $\psi(x_0)\neq 0$, because, due to the Beurling theorem \cite{beurling}, $\mathcal{E}_{\ast}(\Omega)$ is inverse closed.

Let now $f = [(f_\varepsilon)_\varepsilon] \in \mathcal{G}_\ast(\Omega)$. The wave front $WF_{g,\ast}(f)$ is defined as the complement in $\Omega \times ( \R^d \backslash \{ 0 \})$ of the set pairs $(x_0, \xi_0) $ for which there are an open conic neighborhood $\Gamma$ of $\xi_0$ and $\psi \in \mathcal{D}_\ast(\Omega)$ with $\psi(x_0) \neq 1$  such that
\[( \exists k > 0)(\forall \lambda > 0) \qquad  (( \exists \lambda > 0)(\forall k > 0))    \]
\[ \sup_{\xi \in \Gamma}|\widehat{\psi f}_\varepsilon(\xi)|e^{\lambda \omega(|\xi|)} = O(e^{k\omega(1/\varepsilon)}). \]
This definition is clearly independent of the representative of $f$.  Versions of all of our results from Section \ref{section wavefronts} remain valid for the generalized wave front set $WF_{g,\ast}$ with respect to $\ast=(\omega)$ or$ \{\omega\}$. In particular, 
$$
 WF_\ast(f) = WF_{g,\ast}(\iota(f)), \qquad \forall f \in \mathcal{D}'_\ast(\Omega). 
$$
Once again we omit details and leave to the reader the formulations and proofs of the other results from Section \ref{section wavefronts}.

\end{document}